\documentclass[11pt,reqno]{amsart}
\usepackage{amssymb}
\usepackage{mathptmx}
\usepackage{genyoungtabtikz}
\usepackage{mathtools}
\usepackage{mathrsfs}
\usepackage[export]{adjustbox}
\usepackage[all]{xy}
\usepackage{stmaryrd}
\usepackage{fancyhdr}
\usepackage{hyperref}
\usepackage{tikz-cd}
\usepackage{cite}
\usepackage{amsmath,amsfonts}
\usepackage{graphicx}
\usepackage{wrapfig}
\usepackage{array}
\usepackage{amsthm}
\usepackage[left=1.2in, right=1.2in, top=1in, bottom=1in]{geometry}
\usepackage{etoolbox}
\patchcmd{\section}{\scshape}{\bfseries}{}{}
\makeatletter
\renewcommand{\@secnumfont}{\bfseries}
\makeatother
\xyoption{all}
\usepackage{secdot}

\theoremstyle{plain}
\DeclareMathOperator{\id}{\textrm{id}}

\newcommand{\FF}{\mathbb{F}}    
\newcommand{\NN}{\mathbb{N}} 
\newcommand{\ZZ}{\mathbb{Z}}
\newcommand{\wild}{\mathbb{L}}% I'm going to temporarily use this command to denote the n-loop quiver, if we think of a snappier name for it then it should be easy to change.  
\newcommand{\I}{\operatorname{I}} %I'm going to temporarily use this command to denote the function which counts indecomposables of a given dimension.
\newcommand{\NI}{\operatorname{NI}} %I'm going to temporarily use this command to denote the function which counts nilpotent indecomposables of a given dimension.
 
\newcommand{\Rep}{\operatorname{Rep}}  
\newcommand{\Res}{\operatorname{Res}}
\newcommand{\nil}{\operatorname{nil}}
\newcommand{\Hom}{\operatorname{Hom}} 
\newcommand{\Spec}{\operatorname{Spec}} 

\input xy
\xyoption{all}

\theoremstyle{definition}
\newtheorem{mydef}{\textbf{Definition}}[section]
\newtheorem{myeg}[mydef]{\textbf{Example}}

\newtheorem{question}[mydef]{\textbf{Question}}

\newtheorem{rmk}[mydef]{\textbf{Remark}} 
\newtheorem{conv}[mydef]{\textbf{Convention}}
\newtheorem{construction}[mydef]{\textbf{Construction}}
\newtheorem*{notation}{\textbf{Notation}}

\theoremstyle{plain}

\newtheorem*{nothma}{\textbf{Theorem A}}
\newtheorem*{nothmb}{\textbf{Theorem B}}
\newtheorem*{nothmc}{\textbf{Theorem C}}
\newtheorem*{nothmd}{\textbf{Theorem D}}
\newtheorem*{nothme}{\textbf{Theorem E}}
\newtheorem*{nothmf}{\textbf{Theorem F}}

\newtheorem{mytheorem}[mydef]{\textbf{Theorem}}
\newtheorem{mythm}[mydef]{\textbf{Theorem}}
\newtheorem{lem}[mydef]{\textbf{Lemma}}
\newtheorem{pro}[mydef]{\textbf{Proposition}}

\newtheorem{cor}[mydef]{\textbf{Corollary}}

\tikzset{main node/.style={circle,fill=black,draw,minimum size=0.3cm,inner sep=0pt},
}

\begin{document}

	\title{On quiver representations over $\mathbb{F}_1$}
	%\title{The Hopf structure of the algebra of graphs}
	%\title{Primitive elements in the graph Hopf algebra}

	\author{Jaiung Jun}
	\address{State University of New York at New Paltz, NY, USA}
	\curraddr{}
	\email{junj@newpaltz.edu}
	
	\author{Alexander Sistko}
	\address{Manhattan College, NY, USA}
	\curraddr{}
	\email{asistko01@manhattan.edu}

	\makeatletter
	\@namedef{subjclassname@2020}{%
		\textup{2020} Mathematics Subject Classification}
	\makeatother
	
	\subjclass[2020]{Primary 16G20; Secondary 05E10, 16G60, 17B35}
	\keywords{quiver, the field with one element, representation of quivers, representation type, coefficient quiver, Hall algebra, skew shape}
	%\date{\today}
	
	\dedicatory{{\bf{Corresponding Author:}} Alexander Sistko}

	\maketitle

	%\tableofcontents
	
\begin{abstract}
We study the category $\textrm{Rep}(Q,\FF_1)$ of representations of a quiver $Q$ over ``the field with one element'', denoted by $\FF_1$, and the Hall algebra of $\textrm{Rep}(Q,\FF_1)$. Representations of $Q$ over $\FF_1$ often reflect combinatorics of those over $\FF_q$, but show some subtleties - for example, we prove that a connected quiver $Q$ is of finite representation type over $\FF_1$ if and only if $Q$ is a tree. Then, to each representation $\mathbb{V}$ of $Q$ over $\FF_1$ we associate a coefficient quiver $\Gamma_\mathbb{V}$ possessing the same information as $\mathbb{V}$. This allows us to translate representations over $\FF_1$ purely in terms of combinatorics of associated coefficient quivers. We also explore the growth of indecomposable representations of $Q$ over $\FF_1$ - there are also similarities to representations over a field, but with some subtle differences. Finally, we link the Hall algebra of the category of nilpotent representations of an $n$-loop quiver over $\FF_1$ with the Hopf algebra of skew shapes introduced by Szczesny. 
\end{abstract} 

%%%%%%%%%%%%%%%%%%(INTRODUCTION)%%%%%%%%%%%%%%
	
\section{Introduction}

The idea of the mysterious algebraic structure ``the field of characteristic one'' goes back to Tits \cite{tits1956analogues} where he observed an incidence geometry $\Gamma(\mathbb{F}_q)$ associated to a Chevalley group $G(\mathbb{F}_q)$ over a finite field $\mathbb{F}_q$ does not completely degenerate, whereas the algebraic structure of $\mathbb{F}_q$ completely degenerates as $q \to 1$. Tits suggested that the geometric object $\lim_{q \to 1} \Gamma(\mathbb{F}_q)$ should be a geometry which is defined over ``the field of characteristic one'' and ought to contain a (combinatorial) core of a Chevalley group. In fact, one may also observe the following: let $G$ be a Chevalley group of rank $\ell$ and $W_G$ be the Weyl group of $G$. From the Bruhat decomposition, one has the following counting formula:
\[
|G(\FF_q)|=\sum_{w \in W_G} (q-1)^\ell q^{n_w}, \quad \textrm{for some } n_w \geq 0.
\] 
The Weyl group $W_G$ is arguably a combinatorial core of $G$, and by removing zeros from the counting polynomial of $G(\FF_q)$ at $q=1$, one obtains
\[
\lim_{q\to 1} \frac{|G(\mathbb{F}_q)|}{(q-1)^\ell} = |W_G|.
\]
Another interesting example is the Grassmannian $\textrm{Gr}(k,n)$ ($k$-dimensional subspaces in an $n$-dimensional space). The cardinality of the set $\textrm{Gr}(k,n)(\mathbb{F}_q)$ of $\mathbb{F}_q$-rational points of  $\textrm{Gr}(k,n)$ is given by the formula:
\[
|\textrm{Gr}(k,n)(\mathbb{F}_q)|=\left[ {\begin{array}{c} n \\k\\
	\end{array} } \right]_q
\]
where
\[ [n]_q=q^{n-1} + ... + q + 1, \quad [n]_{q}!=\prod_{i=1}^n [i]_q
,\quad
\left[ {\begin{array}{c} n \\k\\
	\end{array} } \right]_q=\frac{[n]_q!}{[k]_q![n-k]_q!}
\]
The limit $q \to 1$ gives us the number ${n \choose k}$, counting $k$ element subsets from an $n$ element set, which is also the Euler characteristic of $\textrm{Gr}(k,n)(\mathbb{C})$. 

Another motivation searching for $\FF_1$ is completely independent from Tits' viewpoint. It arises in the following context (first appearing in Manin's work \cite{manin1995lectures}): translating the geometric proof of the Weil conjectures from function fields to the case of $\mathbb{Q}$ with the hope to shed some light on the Riemann Hypothesis. The approach followed in these past years by several mathematicians is to enlarge the category of commutative rings and to develop a notion of ``generalized Grothendieck's scheme theory'' in order to realize the scheme $\Spec \mathbb{Z}$ as ``a curve over $\mathbb{F}_1$''. See, for instance, \cite{deitmar2008f1}, \cite{con2}, \cite{con1}, \cite{connes2019absolute}, \cite{Deitmar}, \cite{oliver1}, \cite{pena2009mapping}, \cite{soule2004varietes}, \cite{toen2009dessous}.

Quivers arise naturally from problems in invariant theory, matrix problems and the representation theory of associative algebras \cite{RepTheoryArtinAlgebras}, \cite{ElementsofRepTheory1}. For an algebraically closed field $k$, it is widely known that admissible quotients of path algebras yield the Morita classes of finite dimensional $k$-algebras. In this correspondence, the quiver is unique up to isomorphism and acyclic path algebras yield the Morita classes of hereditary algebras. To understand the representation theory of a quiver $Q$, one needs to understand its indecomposable representations. Gabriel's celebrated theorem tells us that $Q$ has finitely-many isomorphism classes of indecomposables if and only if its underlying graph is a Dynkin diagram of types $\mathbb{A}_n$, $\mathbb{D}_n$, $\mathbb{E}_6$, $\mathbb{E}_7$ or $\mathbb{E}_8$ \cite{Gabriel1972}. We say that such quivers are of \emph{finite representation type}. The finite-representation-type path algebras belong to the class of \emph{tame algebras}, whose indecomposables in each dimension vector can be described by finitely-many one-parameter families. Another fundamental result states that any finite-dimensional algebra $A$ which is not tame is \emph{wild}, in the sense that its module category contains the representations of \emph{any} finite-dimensional algebra as a closed subcategory \cite{Drozd1980}, \cite{CrawleyBoevey1988}, \cite{Simson2005}. Moreover, the following conditions on $A$ are equivalent: 
\begin{enumerate} 
\item $A$ is of wild type; 
\item The category $\operatorname{mod-}A$ contains $n$-parameter families of indecomposables, for every $n \geq 1$; 
\item The category $\operatorname{mod-}A$ contains a $2$-parameter family of indecomposables; 
\item There is a fully faithful exact functor $\operatorname{mod-}k\langle x,y\rangle \rightarrow \operatorname{mod-}A$ which preserves indecomposables (where $k\langle x,y\rangle$ is the free algebra in two variables).
\end{enumerate} 
 The so-called Tame-Wild Dichotomy is often cited as a benchmark for whether one can understand the representation theory of a given algebra, with wild algebras being characterized as ``hopeless.''  

Classically, the Hall algebra of a finite abelian group $G$ is an associative algebra naturally defined by the combinatorics of flags of abelian $p$-subgroups of $G$. In his groundbreaking work \cite{ringel1990hall}, Ringel notices that the classical Hall algebra construction can be applied to the category $\textrm{Rep}(Q,\mathbb{F}_q)$, and proves that in the case of a simply-laced Dynkin quiver $Q$, the resulting associative algebra realizes the upper triangular part of the quantum group classified by the same underlying Dynkin diagram $\overline{Q}$. In fact, to an abelian category $\mathcal{A}$ satisfying a certain finiteness condition, one may associate the Hall algebra $H_\mathcal{A}$; one first considers $H_\mathcal{A}$ as a vector space spanned by the set $\textrm{Iso}(\mathcal{A})$ of isomorphism classes of $\mathcal{A}$. For each $M, N \in \textrm{Iso}(\mathcal{A})$, multiplication is defined as follows:
\[
M\cdot N := \sum_{R \in \textrm{Iso}(\mathcal{A})}\textbf{a}^R_{M,N}R,
\]
where 
\[
\textbf{a}^R_{M,N}=\#\{L \subseteq R \mid L\simeq N\textrm{ and } R/L \simeq M\}.
\]
The existence of a Hopf algebra structure for $H_\mathcal{A}$ further depends on properties of $\mathcal{A}$, and is rather subtle. In his celebrated work \cite{kapranov1997eisenstein}, Kapranov investigates the case when $\mathcal{A}$ is the category of coherent sheaves $Coh(X)$ on a smooth projective curve $X$ over a finite field $\mathbb{F}_q$, and proves when $X=\mathbb{P}^1_{\mathbb{F}_q}$, a certain subalgebra of the Hall algebra $H_{Coh(X)}$ is isomorphic to a ``positive part'' of the quantum affine algebra $U_q(\widehat{\mathfrak{sl}_2})$. However, the structure of $H_{Coh(X)}$ is barely known even for $X=\mathbb{P}^2_{\mathbb{F}_q}$. Finally, we note that in their work \cite{dyckerhoff2012higher}, Dyckerhoff and Kapranov introduce the notion of proto-exact categories which allows one to construct the Hall algebra in a more general setting beyond abelian categories. This framework is well suited to define and study the Hall algebra of a category whose objects are combinatorial. For instance, in \cite{eppolito2018proto}, the Hall algebra of the category of matroids is shown to be isomorphic to the dual of a matroid-minor Hopf algebra. 

The category $\textrm{Rep}(Q,\FF_1)$ of $\FF_1$-representations of a quiver $Q$ and its Hall algebra $H_Q$ are first defined and studied by Szczesny in \cite{szczesny2011representations}. Szczesny's main observation is that in some cases the Hall algebra $H_Q$ of $\textrm{Rep}(Q,\FF_1)$ may behave like the specialization at $q=1$ of the Hall algebra of $\textrm{Rep}(Q,\FF_q)$. For instance, Szczesny proves that the Hall algebra of the category of nilpotent representations of the Jordan quiver over $\FF_1$ is isomorphic to the ring of symmetric functions (as Hopf algebras). %over $\mathbb{Z}$; the Hall algebra of the category of nilpotent representations of the Jordan quiver over $\FF_q$. 
Szczesny also proves that there exists a Hopf algebra homomorphism $\rho':\mathbf{U}(\mathfrak{n}_+) \to H_Q$, where $\mathbf{U}(\mathfrak{n}_+)$ is the enveloping algebra of the nilpotent part $\mathfrak{n}_+$ of the Kac-Moody algebra with the same underlying Dynkin diagram $\overline{Q}$. \par\medskip

The aim of this paper is in line with Tits' original idea and Szczesny's work. We investigate $\textrm{Rep}(Q,\FF_1)$, the category of representations of a quiver $Q$ over $\FF_1$, by considering it as a degenerated combinatorial model of the category $\textrm{Rep}(Q,\FF_q)$. We also study the Hall algebra of $\textrm{Rep}(Q,\FF_1)$ which may retain certain combinatorial information of the Hall algebra of $\textrm{Rep}(Q,\FF_q)$. The novelty of our approach is in the widespread adoption of coefficient quivers as in \cite{Ringel1998}, \cite{Kinser2010}, and the function $\NI_Q$ which measures the growth of the number of nilpotent indecomposable $\FF_1$-representations of $Q$. Taking inspiration from the Tame-Wild Dichotomy, the function $\NI_Q$ hints at a way to stratify quivers according to the complexity of their $\FF_1$-representation theory.

In what follows, by a quiver we will always mean a finite quiver. If $Q$ is a fixed quiver, then a coefficient quiver $(\Gamma,c)$ will be a quiver $\Gamma$ along with a quiver map $c : \Gamma \rightarrow Q$. For computational purposes, it will often be expedient to view $(\Gamma, c)$ as a coloring of the vertices (resp. arrows) of $\Gamma$ by the vertices (resp. arrows) of $Q$. With this understanding, we will sometimes refer to a coefficient quiver as a colored quiver.  To a representation $\mathbb{V}$ of a quiver $Q$ over $\FF_1$, we associate a coefficient quiver $(\Gamma_\mathbb{V},c_\mathbb{V})$. In Lemma \ref{l.basicprop} we prove several properties of coefficient quivers arising from $\FF_1$-representations. In particular, we show that these are precisely the windings $c : \Gamma \rightarrow Q$, i.e. the quiver maps satisfying the following property: if $\alpha$ and $\beta$ are distinct arrows of $\Gamma$, then $s(\alpha) = s(\beta)$ or $t(\alpha) = t(\beta)$ implies $c(\alpha) \neq c(\beta)$. We then prove that this condition precisely characterizes coefficient quivers obtained from representations of $Q$ over $\FF_1$. More explicitly, we prove the following.

\begin{nothma}(Proposition \ref{proposition: colored quiver})
Suppose that $c: \Gamma \rightarrow Q$ is a winding. Then there is an $\FF_1$-representation $\mathbb{V}$ of $Q$ and a quiver isomorphism $\phi: \Gamma \rightarrow \Gamma_\mathbb{V}$ such that $c = c_\mathbb{V} \circ \phi$. Furthermore, $\mathbb{V}$ is well-defined up to isomorphism.
\end{nothma}

In Section \ref{section: n-loop representations}, we initiate a more detailed study of the representations of $\wild_n$, the quiver with one vertex and $n$ loops. As a result of this study, we introduce an order relation $\le_{\nil}$ on quivers based on the growth of the number of their indecomposable representations. To be precise, for a quiver $Q$, we define the function $\NI_Q:\mathbb{N} \to \mathbb{N}$ such that 
\[
 \NI_Q(n) = \#\{\text{isomorphism classes of $n$-dimensional indecomposables in $\Rep(Q,\FF_1)_{\nil}$} \}. 
\]
Note that we follow the convention $\mathbb{N} = \{ 0,1,2,3,...\}$. Then, for quivers $Q$, $Q'$, we define:
\[
Q \le_{\nil} Q' \iff \exists D \in \mathbb{R}_+, C \in \mathbb{N} \textrm{ such that } \NI_Q(n) \leq D\NI_{Q'}(Cn), \forall n \gg 0. 
\]
This order relation induces an equivalence relation $\approx_{\nil}$ on quivers as follows: $Q \approx_{\nil} Q'$ if and only if $Q \le_{\nil} Q'$ and $Q' \le_{\nil} Q$. With this, we prove the following.

\begin{nothmb} (Theorem \ref{theorem: upper bound})
Let $Q$ be a quiver. Then $Q \le_{\nil} \wild_2$.
\end{nothmb}

In Section \ref{section: growh of indecomposable quiver reps}, we proceed to stratify quivers using the relations $\le_{\nil}$ and $\approx_{\nil}$. The position of a quiver within this stratification serves as a measure of its representation-theoretic complexity over $\FF_1$. Several $\approx_{\nil}$-equivalence classes represent properties with direct parallels within the representation theory of quivers over fields. For instance, we say that $Q$ has finite representation type (or is representation finite) over $\FF_1$ if there are only finitely many isomorphism classes of indecomposables in $\Rep(Q,\FF_1)_{\nil}$. We then prove the following:
 
\begin{nothmc}(Theorem \ref{theorem: finite type})
 The following are equivalent for a quiver $Q$.
 \begin{enumerate} 
 \item $Q$ has finite representation type over $\FF_1$. 
 \item $Q$ is a tree. 
 \item $Q\approx_{\nil} \wild_0$. 
 \end{enumerate}
\end{nothmc}

Note that Szczesny proved that trees have finite representation type over $\FF_1$ in \cite{szczesny2011representations}. 

We obtain similar results for quivers that do not have finite representation type over $\FF_1$. We say that $Q$ has bounded representation type over $\FF_1$ if $Q \le_{\nil} \wild_1$, where $\wild_1$ is the Jordan quiver. More explicitly, this means that there exists a positive constant $M$ such that $\NI_Q(n) \le M$ for all $n$. We prove the following characterization of quivers with bounded representation type over $\FF_1$.

\begin{nothmd}(Theorem \ref{theorem: cycle or tree if and only if bdd rep})
Let $Q$ be a connected quiver. Then $Q$ has bounded representation type over $\FF_1$ if and only if $Q$ is either a tree or of type $\tilde{\mathbb{A}}_n$. Moreover, $Q$ is a tree quiver if and only if $Q \approx_{\nil} \wild_0$, and $Q$ is of type $\tilde{\mathbb{A}}_n$ if and only if $Q \approx_{\nil} \wild_1$.
\end{nothmd}

In the above theorem, the equi-oriented cyclic orientation of $\tilde{\mathbb{A}}_n$ is allowed.  

 We also investigate quivers which do not have bounded representation type over $\FF_1$, and prove the following. By a pseudotree, we mean a graph with at most one cycle (or an orientation of such a graph). A proper pseudotree is a pseudotree which is not a tree or an $\tilde{\mathbb{A}}_n$. We then prove the theorem which follows. Note that the categories referenced in this theorem are all proto-exact, so that a suitable notion of short exact sequence can be defined within each. A functor between such categories is said to be exact if it preserves short exact sequences, for details see \cite{eppolito2018proto}. 

\begin{nothme}(Theorem \ref{theorem: not bdd type})
Suppose that $Q$ is a connected quiver that is not of bounded representation type. Then there exists a fully faithful, exact, indecomposable-preserving functor  
\[
\Rep(Q',\FF_1)_{\nil} \rightarrow \Rep(Q,\FF_1)_{\nil}, 
\] 
where $Q'$ is either a proper pseudotree or $\wild_2$. If $Q$ is not a pseudotree, then $Q \approx_{\nil} \wild_2$.	
\end{nothme}

Finally, in Section \ref{section: hall algebra of full subcategories}, we study Hall algebras arising from full subcategories of $\textrm{Rep}(Q,\FF_1)_{\nil}$. For each quiver $Q$, the existence of the Hall algebra $H_{Q}$ is proved by Szczesny in \cite[Theorem 6]{szczesny2014hall}. Let $H_{Q,\nil}$ be the Hall algebra of $\Rep(Q,\FF_1)_{\nil}$.

\begin{nothmf}(Theorem \ref{theorem: hall algebra for $L_n$})
The Hall algebra $H_{Q,\nil}$ is isomorphic to the enveloping algebra $\mathbf{U}(\mathfrak{cq}_n)$ of a graded Lie algebra $\mathfrak{cq}_n$. The Lie algebra $\mathfrak{cq}_n$ has a basis corresponding to connected, acyclic coefficient quivers $(\Gamma,c)$.
\end{nothmf}

We also define the path monoid $M_Q$ of a quiver $Q$ and prove analogous statements of the classical equivalence of categories between $\textrm{Rep}(Q,k)$ and the category of left $kQ$-modules, where $k$ is a field and $kQ$ is the path algebra associated to $Q$. By using this observation, we link the Hall algebra of a certain subcategory of $H_\textrm{$\wild_n$,nil}$ with the Hall algebra of $\textbf{SK}_n$ of skew shapes introduced by Szczesny  in \cite{szczesny2018hopf}  which may be viewed as an $n$-dimensional generalization of the Hopf algebra of symmetric functions.

\bigskip

\textbf{Acknowledgment}\hspace{0.1cm} We would like to thank Ryan Kinser for helpful comments and various suggestions leading us to a upcoming companion of the current paper. We are also grateful to Matt Szczesny for his detailed
feedback and for pointing out some minor mistakes in the first draft (and for suggesting a way to fix them).  

\bigskip

\textbf{Data Sharing Statement}\hspace{0.1cm} Data sharing not applicable to this article as no datasets were generated or analysed during the current study.

%%%%%%%%%%%%%%%%%%%%%(PRELIMINARIES)%%%%%%%%%%%%%%%%%
	
\section{Preliminaries}

\subsection{Representation of quivers over $\mathbb{F}_1$}

In this section, we recall basic definitions and properties concerning representations of quivers over $\FF_1$ which will be used throughout the paper. 

\begin{mydef}
Let $\textrm{Vect}(\mathbb{F}_1)$ be the category whose objects are finite pointed sets $(V,0_V)$, and morphisms are pointed functions $f:V \to W$ such that $f|_{V-f^{-1}(0_W)}$ is an injection. We call $\textrm{Vect}(\mathbb{F}_1)$, the category of \emph{finite dimensional vector spaces} over $\mathbb{F}_1$. 
\end{mydef}

\begin{notation} 
For any natural number $n$, we let $[n]$ be the $\FF_1$-vector space $\{0,1,\ldots , n\}$. 
\end{notation}

In what follows, we will simply refer to objects (resp.~morphisms) in $\textrm{Vect}(\mathbb{F}_1)$ as $\mathbb{F}_1$-vector spaces (resp.~$\mathbb{F}_1$-linear maps). We first recall basic definitions. 

\begin{mydef}\label{definition: $F_1$-vectorspace}
Let $V$ and $W$ be $\mathbb{F}_1$-vector spaces. 
\begin{enumerate}
\item 
The \emph{direct sum} is defined as $V\oplus W:=V\sqcup W /\langle 0_V \sim 0_W\rangle$.
	\item 
By the \emph{dimension} of an $\mathbb{F}_1$-vector space $V$, we mean $\dim(V):=|V|-1$, the number of nonzero elements of $V$.
\item 
There exists a unique $\mathbb{F}_1$-linear map $0: V \to W$ sending any element in $V$ to $0_W$, called the \emph{zero map}.
\item 
$W$ is said to be a \emph{subspace} of $V$ if $W$ is a subset of $V$ containing $0_V$.
\item 
Let $W$ be a subspace of $V$. The \emph{quotient space} $V/W$ is defined to be $V-(W-\{0_V\})$. 
\item 
An endomorphism $f \in \textrm{End}(V)$ of an $\mathbb{F}_1$-vector space $V$ is said to be \emph{nilpotent} if $f^n=0$ for some $n \in \mathbb{N}$. 
\item 
For an $\mathbb{F}_1$-linear map $f:V \to W$, the \emph{kernel} of $f$ is $\ker(f):=f^{-1}(0_W)$.
\item 
For an $\mathbb{F}_1$-linear map $f:V \to W$, the \emph{cokernel} of $f$ is $\textrm{coker}(f):=W/f(V)$.%\footnote{Kernels and cokernels satisfy the usual universal property.}
\end{enumerate}
\end{mydef}

\begin{mydef}
Let $M$ be a monoid (not necessarily commutative) with an absorbing element $0_M$. By a left $M$-module, we mean a pointed set $S$  with a map $\cdot : M \times S \to S$ satisfying the following axioms: 
\begin{enumerate} 
\item $(ab)\cdot s = a\cdot (b\cdot s)$ for all $a, b \in M$ and $s\in S$.
\item $1_M\cdot s = s$ for all $s \in S$.
\item $0_M\cdot s = 0_S$, where $0_S$ is the distinguished element of $S$.
\end{enumerate}  
\end{mydef}

\begin{mydef}
A \emph{quiver} $Q$ is a finite directed graph (with possibly with multiple arrows and loops). We denote a quiver $Q$ as a quadruple $Q=(Q_0,Q_1,s,t)$;
\begin{enumerate}
	\item 
$Q_0$ and $Q_1$ are finite sets; $Q_0$ (resp.~$Q_1$) is the set of vertices (resp.~arrows),
\item
$s$ and $t$ are functions
\[
s,t:Q_1 \to Q_0
\] 
assigning to each arrow in $Q_1$ its \emph{source} and \emph{target} in $Q_0$. For each arrow $\alpha \in Q_1$, for the notational convenience, we let $s(\alpha)=\alpha_s$ and $t(\alpha)=\alpha_t$.
\end{enumerate}
\end{mydef}

We will simply denote a quiver by $Q$ or $Q=(Q_0,Q_1)$. We say that $Q$ is \emph{connected} if its underlying undirected graph is connected. We say that $Q$ is \emph{acyclic} if it does not contain any oriented cycles. 

Let $Q$ and $Q'$ be quivers. A \emph{quiver map} $f : Q \rightarrow Q'$ is a pair of functions  
\[ 
f_i : Q_i \rightarrow Q_i'
\] 
for $i = 0,1$ satisfying  
\[
s(f_1(\alpha)) = f_0(s(\alpha)) 
\] 
and  
\[t(f_1(\alpha)) = f_0(t(\alpha)), 
\]
for all $\alpha \in Q_1$. A quiver map $f$ is injective (resp.~surjective) if and only if both $f_0$ and $f_1$ are injective (resp.~surjective). 
By a \emph{subquiver} $S$ of $Q$, we simply mean a quiver $S = (S_0,S_1)$ such that $S_i \subseteq Q_i$ for $i = 0,1$. Of course, we can identify $S$ with the image of the obvious inclusion map $S \hookrightarrow Q$. We say that $S$ is \emph{full} if for any two vertices $u, v \in S_0$, any arrow in $Q$ from $u$ to $v$ (and from $v$ to $u$) belongs to $S_1$.   

\begin{mydef} 
For $n\geq 0$, we let $\wild_n$ denote the quiver with one vertex and $n$ loops. When $n=1$ this is known as the \emph{Jordan quiver}.
\end{mydef}

\begin{rmk} 
We denote the underlying graph of a quiver $Q$ by $\overline{Q}$. Throughout this paper, we will freely apply the basic terminology of undirected graphs to $Q$. This should be interpreted to mean that the corresponding graph-theoretic property holds for $\overline{Q}$. For instance, we will say that $Q$ is a tree, of type $\tilde{\mathbb{A}}_n$, or connected if $\overline{Q}$ is a tree, of type $\tilde{\mathbb{A}}_n$, or connected etc. 
\end{rmk}

\begin{mydef}\cite[Definition 4.1]{szczesny2011representations}\label{definition: representation of a quiver over $F_1$}
Let $Q$ be a quiver. By a representation of $Q$ over $\FF_1$ (or an $\FF_1$-representation of $Q$), we mean the collection of data $\mathbb{V}=(V_i,f_\alpha)$, $i\in Q_0$, $\alpha \in Q_1$: 
\begin{enumerate}
\item 
An assignment of an $\mathbb{F}_1$-vector space $V_i$ for each vertex $i \in Q_0$. 
\item 
An assignment of an $\mathbb{F}_1$-linear map $f_\alpha \in \Hom(V_{\alpha_s},V_{\alpha_t})$ for each arrow $\alpha \in Q_1$. 
\end{enumerate} 
\end{mydef}

\begin{mydef}\cite[Definition 4.3]{szczesny2011representations}
Let $Q$ be a quiver and $\mathbb{V}=(V_i,f_\alpha)$ be a representation of $Q$ over $\mathbb{F}_1$. 
\begin{enumerate}
	\item 
The \emph{dimension} of $\mathbb{V}$ is defined by:
\[
\dim(\mathbb{V})=\sum_{i \in Q_0} \dim(V_i).
\]
\item 
The \emph{dimension vector} of $\mathbb{V}$ is the $|Q_0|$-tuple:
\[
\underline{\dim}(\mathbb{V})=(\dim(V_i))_{i \in Q_0}. 
\]
\end{enumerate}
A representation $\mathbb{V}=(V_i,f_\alpha)$ is \emph{nilpotent} if there exists a positive integer $N$ such that for any $n \geq N$ and any path $\alpha_1\alpha_2\dots \alpha_n$ in $Q$ (left-to-right in the order of traversal), we have
\[
f_{\alpha_n}f_{\alpha_{n-1}}\cdots f_{\alpha_1}=0.
\]
\end{mydef}

For representations  $\mathbb{V}=(V_i,f_\alpha)$ and $\mathbb{W}=(W_i,g_\alpha)$ of a quiver $Q$ over $\FF_1$, a morphism $\Phi:\mathbb{V} \to \mathbb{W}$ is a collection of $\FF_1$-linear maps $(\phi_i)_{i \in Q_0}$, where $\phi_i:V_i \to W_i$, such that the following commutes for each $i \in Q_0$: 
\begin{equation}
\begin{tikzcd}[row sep=large, column sep=1.5cm]
V_{\alpha_s}\arrow{r}{\phi_{\alpha_s}}\arrow{d}{f_\alpha}
& W_{\alpha_s} \arrow{d}{g_\alpha} \\
V_{\alpha_t} \arrow{r}{\phi_{\alpha_t}} 
& W_{\alpha_t}
\end{tikzcd}
\end{equation}
This defines the category $\textrm{Rep}(Q,\mathbb{F}_1)$ of representations of $Q$ over $\mathbb{F}_1$. We let $\textrm{Rep}(Q,\mathbb{F}_1)_{\textrm{nil}}$ be the full subcategory of $\textrm{Rep}(Q,\mathbb{F}_1)$ consisting of nilpotent representations. 

Each morphism $\Phi \in \Hom(\mathbb{V},\mathbb{W})$ has a kernel and cokernel defined in a component-wise way by using Definition \ref{definition: $F_1$-vectorspace}. Similarly, one defines sub-representations and quotient representations. See \cite[Definition 4.3]{szczesny2011representations} for details.

Recall that an $\FF_1$-representation $\mathbb{V}$ of a quiver $Q$ is \emph{indecomposable} if $\mathbb{V}$ cannot be written as a nontrivial direct sum of sub-representations 

\begin{mydef}
We say that a quiver $Q$ has \emph{finite representation type} over $\FF_1$ if there are finitely many isomorphism classes of indecomposables in $\Rep (Q,\FF_1)_{\nil}$. 
\end{mydef}

\begin{rmk}
Let $Q$ be a quiver. In \cite{szczesny2011representations}, Szczesny proves that the Krull-Schmidt Theorem holds for $\Rep(Q,\FF_1)$: any object $M$ can be written uniquely (up to permutation) as a finite direct sum $M = M_1 \oplus \cdots \oplus M_k$ of indecomposable representations. The same statement holds for $\Rep(Q,\FF_1)_{\nil}$. 
\end{rmk}

Let $Q$ be a quiver. For any dimension $d$, $Q$ has finitely many isomorphism classes of $d$-dimensional representations over $\FF_1$. To see this, fix a dimension vector $\underline{d}$ and consider the following finite set:
\[
\Rep_{\underline{d}}(Q,\FF_1):=\prod_{\alpha \in Q_1}{\Hom_{\FF_1}({[\underline{d}(\alpha_s)]}, [{\underline{d}(\alpha_t)}])}. 
\]
 Any representation of $Q$ with dimension vector $\underline{d}$ can be identified with a point of $\Rep_{\underline{d}}(Q,\FF_1)$ by fixing an ordering on the nonzero elements at each vertex. The group  
 \[ 
 \operatorname{GL}_{\underline{d}}(\FF_1) = \prod_{v \in Q_0}{\operatorname{Aut}_{\FF_1}([{\underline{d}(v)}])}
 \] 
 acts on $\Rep_{\underline{d}}(Q,\FF_1)$ via  
 \[
 (\phi_v)_{v \in Q_0}\cdot (f_{\alpha})_{\alpha \in Q_1} =  \left(\phi_{t(\alpha)}f_{\alpha}\phi_{s(\alpha)}^{-1}\right)_{\alpha \in Q_1},
 \] 
 and it is easy to check that two points in $\Rep_{\underline{d}}(Q,\FF_1)$ correspond to isomorphic representations if and only if they lie in the same $\operatorname{GL}_{\underline{d}}(\FF_1)$-orbit. Since  $\Rep_{\underline{d}}(Q,\FF_1)$ is a finite set, there are finitely many isomorphism classes of $\underline{d}$-dimensional representations. Since there are only finitely-many dimension vectors $\underline{d}$ with $d = \sum_{i \in Q_0}{\underline{d}(i)}$, there are finitely many isomorphism classes of $d$-dimensional representations of $Q$ for any natural number $d$. Similar statements hold for nilpotent representations. We can therefore count the number of (nilpotent) indecomposable representations of any dimension. The following is a key definition in this paper. 
 \begin{mydef} \label{definition: indecomposables growth}
 Let $Q$ be a quiver.
 \begin{enumerate} 
 \item  $\I_Q : \NN \rightarrow \NN$ is the function such that
 \[ 
 I_Q(n) = \#\{\text{isomorphism classes of $n$-dimension indecomposables in $\Rep(Q,\FF_1)$} \}.
 \]
\item $\NI_Q : \NN \rightarrow \NN$ is the function such that
 \[
 \NI_Q(n) = \#\{\text{isomorphism classes of $n$-dimensional indecomposables in $\Rep(Q,\FF_1)_{\nil}$} \}. 
 \] 
 \end{enumerate}
 \end{mydef} 
Of course, $Q$ is of finite representation type if and only if $\NI_Q(n) = 0$ for $n \gg 0$. Similar functions have been considered for representations over finite fields, see for instance \cite{Kac1982}.

\subsection{The Hall algebra of $\textrm{Rep}(Q,\mathbb{F}_1)$}

In this section, we briefly recall the Hall algebras of $\textrm{Rep}(Q,\mathbb{F}_1)$ and $\textrm{Rep}(Q,\mathbb{F}_1)_{\textrm{nil}}$ for a quiver $Q$. We refer the interested reader to \cite[Section 6]{szczesny2011representations} for details. 

Let $Q$ be a quiver, and let $\textrm{Iso}(Q)$ be the set of isomorphism classes of objects in $\textrm{Rep}(Q,\mathbb{F}_1)$. The Hall algebra $H_Q$ of $\textrm{Rep}(Q,\mathbb{F}_1)$ has the following underlying set:
\begin{equation}\label{eq: hall algebra}
H_Q:=\{f:\textrm{Iso}(Q) \to \mathbb{C} \mid \#\{\textrm{supp}(f)\}< \infty\},
\end{equation}
where $\textrm{supp}(f)=\{[M] \in \textrm{Iso}(Q) \mid f([M])\neq 0\}.$ For each $[M] \in \textrm{Iso}(Q)$, we let $\delta_{[M]}$ be the delta function in $H_Q$ supported at $[M]$. Then, we may consider $H_Q$ as the $\mathbb{C}$-vector space spanned by $\{\delta_{[M]}\}_{[M] \in \textrm{Iso}(Q)}$. In what follows, by abuse of notation, we will simply denote the delta function $\delta_{[M]}$ by $[M]$. One defines the following multiplication on delta functions:
\begin{equation}\label{eq: hall product}
[M]\cdot[N]:=\sum_{R \in \textrm{Iso}(Q)} \frac{\mathbf{P}^R_{M,N}}{a_Ma_N}[R],
\end{equation}
where $a_M=|\textrm{Aut}(M)|$ and $\mathbf{P}^R_{M,N}$ is the number of short exact sequences of the form:\footnote{By a short exact sequence, we mean that ``$\ker=\textrm{im}$'' as in the classical case.}
\[
0 \to N \to R \to M \to 0.
\]
Then, one has the following equality:
\[
\textbf{a}^R_{M,N}:=|\{L \subseteq R \mid L\simeq N\textrm{ and } R/L \simeq M\}| = \frac{\mathbf{P}^R_{M,N}}{a_Ma_N}.
\]
By linearly extending \eqref{eq: hall product} to $H_Q$, one obtains an associative $\mathbb{C}$-algebra $H_Q$. %with a canonical grading by $K_0(\textrm{Rep}(Q,\mathbb{F}_1))^+$. 
Furthermore, $H_Q$ is also equipped with the following coproduct:
\begin{equation}\label{eq: hall coprod}
\Delta:H_Q\to H_Q\otimes_\mathbb{C}H_Q, \quad \Delta(f)([M],[N])=f([M\oplus N]).
\end{equation}

With \eqref{eq: hall product} and \eqref{eq: hall coprod}, Szczesny proves the following:

\begin{mythm}\cite[Theorem 6]{szczesny2011representations}\label{theorem: Szczesny Hall algebra}
With the same notation as above, $H_Q$ is a graded, connected, and co-commutative Hopf algebra. In particular, by the Milnor-Moore theorem, $H_Q\simeq \mathbf{U}(\mathfrak{n}_Q)$, where $\mathfrak{n}_Q$ is the pro-nilpotent Lie algebra spanned by indecomposables $[M] \in \textrm{Iso}(Q)$.
\end{mythm}

One can apply a similar construction to the full subcategory $\textrm{Rep}(Q,\mathbb{F}_1)_{\textrm{nil}}$ of $\textrm{Rep}(Q,\mathbb{F}_1)$ to obtain a Hopf algebra $H_\textrm{Q,nil}$; see \cite[Remark 2]{szczesny2011representations}. We will study $H_\textrm{Q,nil}$ in Section \ref{section: hall algebra of full subcategories} by using \emph{coefficient quivers} to be introduced in Section \ref{section: coefficient quivers}.

\subsection{Base change functors}

In many cases, there is a functor which allows one to perform ``base-change'' from $\mathbb{F}_1$-objects to classical objects over a field $k$. By abuse of notation, we will denote these functors by $k\otimes_{\FF_1}-$ when there is no potential confusion. Here are two typical examples of base change functors to be used in this paper.

\begin{myeg}
Let $k$ be a field and $V$ be a vector space over $\FF_1$. We let $V_k$ be the vector space whose basis is $V\backslash \{0_V\}$. For an $\FF_1$-linear map $f:V \to W$, we let $f_k$ be the linear from $V_k$ to $W_k$ induced by $f$. This defines a functor as follows:
\begin{equation}\label{eq: base change vec}
k\otimes{\FF_1}-: \textrm{Vect}(\mathbb{F}_1) \to 	\textrm{Vect}(k), \quad V\mapsto V_k.
\end{equation}
Through the above base chance functor, representations of a quiver can be defined in a more categorical way as follows. Let $Q$ be a quiver. One can consider the free category associated to $Q$; objects are vertices of $Q$ and morphisms are paths. The identity morphisms correspond to the ``stationary paths'' of length zero at each vertices. Then, a representation $M$ of $Q$ over $\FF_1$ is nothing but a functor $\mathbf{M}: \mathcal{Q} \to \textrm{Vect}(\mathbb{F}_1)$. In particular, $\textrm{Rep}(Q,\FF_1)$ is equivalent to the functor category $\textrm{Vect}(\mathbb{F}_1)^{\mathcal{Q}}$. In fact, the same description holds for representations of $Q$ over a field $k$. Therefore, from the base change functor \eqref{eq: base change vec}, one has the following base change functor which is exact and faithful (but not full in general):
\begin{equation}\label{eq: base change}
	k\otimes_{\FF_1} -:\textrm{Rep}(Q,\FF_1) \to \textrm{Rep}(Q,k).
\end{equation}	
%This categorical perspective is sometimes quite useful. For instance, the Hall algebra in Theorem \ref{theorem: Szczesny Hall algebra} can be also constructed as follows: notice that the category $\Vect(\FF_1)$ is proto-exact in the sense of Dyckerhoff and Kapranov \cite{dyckerhoff2012higher}, where they also prove that for a small category $\mathcal{I}$ the functor category $\mathcal{C}^\mathcal{I}$ is proto-exact for a proto-exact category $\mathcal{C}$. Now, from the equivalence $\textrm{Rep}(Q,\FF_1)\simeq \textrm{Vect}(\mathbb{F}_1)^{\mathcal{Q}}$, one obtains Szczesny's Hall algebra in \cite{dyckerhoff2012higher} as the Hall algebra associated a proto-exact category $\textrm{Rep}(Q,\FF_1)$.
\end{myeg}

The following is another example of base-change functors which will be used in \S \ref{section: hall algebra of full subcategories}.

\begin{myeg}\label{example: adjunction for scalar extensions}
Let $\mathbf{Mon}$ be the category of monoids (not necessarily commutative), $k$ a field, and $\mathbf{Alg}_k$ the category of $k$-algebras. Then, one has the following functor:
\[
k\otimes_{\mathbb{F}_1}-: \mathbf{Mon} \to \mathbf{Alg}_k, \quad M \mapsto k[M],
\]
where $k[M]$ is the monoid algebra of $M$ over $k$. We also have an obvious forgetful functor
\[
\mathcal{U}: \mathbf{Alg}_k \to \mathbf{Mon}, \quad A \mapsto (A,\times).
\]
One can easily see that $k\otimes_{\mathbb{F}_1}-$ is a left adjoint of $\mathcal{U}$. 
\end{myeg} 

%%%%%%%%%%%%%%%(COEFFICIENT QUIVERS)%%%%%%%%%%%%%%%%%%%  
\section{The coefficient quiver of a representation}\label{section: coefficient quivers}  

In this section, we show how to associate a coefficient quiver to each object in $\Rep(Q,\FF_1)_{\nil}$. A similar idea for $Q = \wild_1$ was explored in \cite{szczesny2014hall}. This concept will prove essential for the results in following sections, as it reduces problems in $\Rep(Q,\FF_1)$ to purely combinatorial problems about coefficient quivers. 

Within the realm of quiver representations, similar constructions are widespread. For instance, our coefficient quiver will be a coefficient quiver in the sense of \cite{Ringel1998}, and it will yield a \emph{quiver over $Q$} as in \cite{Kinser2010}. The authors would like to thank Ryan Kinser for alerting us to these connections. 

\begin{mydef}\label{definition: coefficient quiver}
Let $\mathbb{V}$ be an $\FF_1$-representation of $Q$. First, we associate to $\mathbb{V}$ a quiver $\Gamma_{\mathbb{V}}$, whose vertex set is  
\[
(\Gamma_{\mathbb{V}})_0 = \bigsqcup_{v \in Q_0}{(\mathbb{V}_v\setminus\{0\})}.
\]

For each $\alpha \in Q_1$, we draw an arrow $(\alpha,i,j)$ in $\Gamma_{\mathbb{V}}$ from $i$ to $j$  if and only if $\mathbb{V}_{\alpha}(i) = j$. We then associate to $\Gamma_{\mathbb{V}}$ a quiver map $c_{\mathbb{V}} : \Gamma_{\mathbb{V}} \rightarrow Q$. This map is defined on vertices via the formula
\[ 
c_{\mathbb{V}}(i) = v,
\]
for all $i \in \mathbb{V}_v$. It is defined on arrows via the formula 
\[ 
c_{\mathbb{V}}(\alpha,i,j) = \alpha.
\]
The ordered pair $(\Gamma_{\mathbb{V}},c_{\mathbb{V}})$ is called the \emph{coefficient quiver} of $\mathbb{V}$.
\end{mydef}   

\begin{rmk} 
Definition \ref{definition: coefficient quiver} is equivalent to applying the coefficient quiver construction of \cite{Ringel1998} to $k\otimes_{\FF_1}\mathbb{V}$, where $k$ is any field. In turn, the map $c_{\mathbb{V}}: \Gamma_{\mathbb{V}} \rightarrow Q$ is a quiver over $Q$ as in \cite{Kinser2010}.
\end{rmk} 

\begin{conv}\label{remark: coloring convention}
In practice it is often cumbersome to work with the map $c_{\mathbb{V}} : \Gamma_{\mathbb{V}} \rightarrow Q$ directly. Instead, it is more convenient to view the coefficient quiver of $\mathbb{V}$ as a coloring of the vertices (resp. arrows) of $\Gamma_{\mathbb{V}}$ by the vertices (resp. arrows) of $Q$. When we wish to emphasize this viewpoint, we will speak of $\Gamma_{\mathbb{V}}$ as a \emph{colored quiver}. We will adopt some terminology specific to this viewpoint throughout. In particular: if $v \in Q_0$ and $\alpha \in Q_1$, we will refer to a vertex $x$  of $\Gamma_{\mathbb{V}}$ as \emph{$v$-colored}  if $c_{\mathbb{V}}(x) = v$, an arrow $\beta$ as \emph{$\alpha$-colored} if resp. $c_{\mathbb{V}}(\beta) = \alpha$, and a path $p = \beta_1\cdots \beta_t$ as \emph{$\alpha$-colored} if $c(\beta_i) = \alpha$ for all $i$. Of course, we also refer to $v$ as the color of $x$, and $\alpha$ as the color of $\beta$. This convention also simplifies several statements in this work, which would otherwise need to be formulated in terms of the fibers of $c_{\mathbb{V}}$.
\end{conv}  

In subsequent sections we will want to glue two coefficient quivers at a vertex. The following definition formalizes this process.

\begin{mydef}\label{definition: gluing}
Let $c : \Gamma \rightarrow Q$ and $c' : \Gamma' \rightarrow Q$ be two quiver maps. Suppose that $v \in \Gamma_0$ and $v' \in \Gamma'_0$ satisfy $c(v) = c'(v')$. We define a new quiver $\Gamma\sqcup_{v\sim v'}\Gamma'$ whose vertices are the elements of $(\Gamma_0)\setminus\{v\}$, the elements of $(\Gamma'_0)\setminus\{v'\}$ and a new element $g$. To define the arrows of this quiver, consider the map  
\[
q_0 : \Gamma_0 \sqcup \Gamma'_0 \rightarrow (\Gamma \sqcup_{v\sim v'}\Gamma')_0
\] 

\noindent specified by  
\[
q_0(v) = q_0(v') = g
\]
 and  
 \[
 q_0(x) = x,\text{ for all $x \in \Gamma_0\setminus\{v\}\sqcup \Gamma'_0\setminus\{v'\}$.} 
 \]
 \noindent Then for each $\alpha \in \Gamma_1 \sqcup \Gamma'_1$ we draw an arrow $q_1(\alpha)$ in $\Gamma\sqcup_{v\sim v'}\Gamma'$ from $q_0(s(\alpha))$ to $q_0(t(\alpha))$. The ordered pair $q = (q_0,q_1)$ is then a quiver map  
 \[
 q : \Gamma \sqcup \Gamma' \rightarrow \Gamma\sqcup_{v\sim v'}\Gamma'. 
 \]
 We say that $\Gamma\sqcup_{v\sim v'}\Gamma'$ is the \emph{amalgam of $\Gamma$ and $\Gamma'$ along $v\sim v'$}. Since $c(v) = c'(v')$, there is a unique map $c\sqcup_{v\sim v'}c'$ making 
 \begin{equation*}
\begin{tikzcd}[row sep=2em]
\Gamma\sqcup\Gamma' \arrow[rr,"q"] \arrow[dr,swap,"c\sqcup c'"]
&& \Gamma\sqcup_{v\sim v'}\Gamma' \arrow[dl,,"c\sqcup_{v\sim v'}c'"] \\
& Q
\end{tikzcd}
\end{equation*}
commute, where $c\sqcup c'$ denotes the usual disjoint union of maps. We say that  
\[ 
c\sqcup_{v\sim v'}: \Gamma\sqcup_{v\sim v'}\Gamma' \rightarrow Q
\] 
is \emph{obtained by gluing $\Gamma$ and $\Gamma'$ along $v\sim v'$}. Note that $(\Gamma\sqcup_{v\sim v'}\Gamma',c\sqcup_{v\sim v'})$ is not necessarily the coefficient quiver for an $\FF_1$-representation of $Q$, although in many instances it will be.
\end{mydef} 

\begin{myeg}
Let $V_0=\{0_{V_0},1,2,3\}$. Consider the following representation $\mathbb{V}=(V_0,f_1,f_2)$ of $\wild_2$:
\[
\begin{tikzcd}
\bullet \arrow[loop left,looseness=20,"f_1"]
\arrow[loop right, looseness=20,"f_2"]
\end{tikzcd}
\]
where $f_1(n)=n-1$ for $n \in \{1,2,3\}$ and $f_2(n)=n-2$ for $n \in \{2,3\}$ and $f_2(1)=0_{V_0}$. Denote the arrow of $\wild_2$ corresponding to $f_1$ by $\alpha_1$, and the arrow corresponding to $f_2$ by $\alpha_2$. Then the coefficient quiver $\Gamma_{\mathbb{V}}$ is the following:
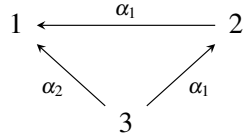
\[
\begin{tikzcd}[arrow style=tikz,>=stealth,row sep=2em]
1 
&& 2 \arrow[ll,swap, "\alpha_1"]\\
& 3 
\arrow[ul,  "\alpha_2"]
\arrow[ur,swap,"\alpha_1"]
\end{tikzcd}
\]
With the representation $\mathbb{W}=(V_0,g_1,g_2)$, where $g_1=f_2$ and $g_2(n)=n+1$ for $n \in \{1,2\}$ and $g_2(3)=0_{V_0}$, the coefficient quiver $\Gamma_{\mathbb{W}}$ is the following:
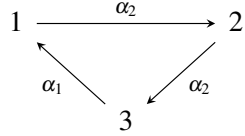
\[
\begin{tikzcd}[arrow style=tikz,>=stealth,row sep=2em]
1 \arrow[rr, "\alpha_2"]
&& 2 \arrow[dl, "\alpha_2"]\\
& 3 
\arrow[ul,"\alpha_1"]
\end{tikzcd}
\] 
Note that we label the arrows by their images under the maps $c_{\mathbb{V}}$ and $c_{\mathbb{W}}$, as in Convention \ref{remark: coloring convention}. We do not label the vertices by their images, since they all get mapped to the unique vertex of $\wild_2$.
\end{myeg} 

\begin{mydef} 
Let $c : \Gamma \rightarrow Q$ be a quiver map. 
\begin{enumerate} 
\item The map $c$ is a \emph{winding} if for any distinct $\beta, \gamma \in \Gamma_1$, $s(\beta) = s(\gamma)$ or $t(\beta ) = t(\gamma)$ implies $c(\beta) \neq c(\gamma)$. Viewing $\Gamma$ as a colored quiver, this means that $\Gamma$ contains no subquivers of the form 
\[
\bullet \xrightarrow[]{\alpha} \bullet \xleftarrow[]{\alpha} \bullet\text{ or } \bullet \xleftarrow[]{\alpha} \bullet \xrightarrow[]{\alpha} \bullet,
\] 
where $\alpha \in Q_1$.
\item Let $\alpha \in Q_1$. A vertex $v \in \Gamma_0$ is an \emph{$\alpha$-sink} if there is no arrow $\beta \in \Gamma_1$ with $s(\beta) = v$ and $c(\beta) = \alpha$. Similarly, we say that $v$ is an \emph{$\alpha$-source} if there is no arrow $\beta \in \Gamma_1$ with $t(\beta) = v$ and $c(\beta) = \alpha$. Viewing $\Gamma$ as a colored quiver, $v$ is an $\alpha$-sink (resp. $\alpha$-source) if no $\alpha$-colored arrow starts at $v$ (resp. terminates at $v$).  
\end{enumerate}
\end{mydef} 

\noindent The following lemma records some basic properties of the quiver $\Gamma_\mathbb{V}$. % when $\mathbb{V}$ is nilpotent.

\begin{lem}\label{l.basicprop}
%Let $\mathbb{V} = (V_u, f_{\alpha})$ be an object in $\Rep(Q,\FF_1)_{\nil}$.  

Then the following statements about $\Gamma_\mathbb{V}$ hold: 
\begin{enumerate}  
\item $c_{\mathbb{V}}:\Gamma_{\mathbb{V}} \rightarrow Q$ is a winding.   
\item Suppose that $x$ is a $v$-colored vertex of $\Gamma_{\mathbb{V}}$. Then the outdegree (resp. indegree) of $x$ in $\Gamma_{\mathbb{V}}$ is bounded above by the outdegree (resp. indegree) of $v$ in $Q$.
\item For each $\alpha \in Q_1$, the number of $\alpha$-sources equals the number of $\alpha$-sinks.   
\item $\mathbb{V}$ is nilpotent if and only if $\Gamma_{\mathbb{V}}$ is acyclic.
\item $\mathbb{V}$ is indecomposable if and only if $\Gamma_\mathbb{V}$ is connected.
\end{enumerate}
\end{lem}    

\begin{proof}   
$(1)$: For $\alpha \in Q_1$, let $S_{\alpha}$ denote the subquiver of $\Gamma_{\mathbb{V}}$ with the same vertex set, whose arrows are precisely the $\alpha$-colored arrows of $\Gamma_{\mathbb{V}}$. By the classification of $\FF_1$-linear endomorphisms in \cite[Lemma 3.1]{szczesny2011representations}, the connected components of $S_{\alpha}$ are either isolated vertices, equioriented quivers of type $\mathbb{A}_d$ or equioriented quivers of type $\tilde{\mathbb{A}}_d$. In particular, $S_{\alpha}$ has no subquivers of the form  
\[
\bullet \xrightarrow[]{\alpha} \bullet \xleftarrow[]{\alpha} \bullet\text{ or } \bullet \xleftarrow[]{\alpha} \bullet \xrightarrow[]{\alpha} \bullet,
\]
 from which the claim follows. 

$(2)$: If $x$ is a $v$-colored vertex and $\alpha \in Q_1$, then by (1) at most one $\alpha$-colored arrow starts at $v$ (resp. ends at $v$). However, $x$ can be the source (resp. target) of an $\alpha$-colored arrow if and only if $v$ is the source (resp. target) of $\alpha$ in $Q$.

$(3)$: This follows from (1) by considering the connected components of each $S_{\alpha}$. Indeed, any vertex of $\Gamma_\mathbb{V}$ is contained in a unique connected component of $S_{\alpha}$. If the connected component is acyclic, then it is an oriented path (possibly of length $0$) which starts at a unique $\alpha$-source and ends at a unique $\alpha$-sink. Otherwise the connected component is of type $\tilde{\mathbb{A}}_d$ and contributes no $\alpha$-sources or $\alpha$-sinks. 

$(4)$: If $\Gamma_\mathbb{V}$ has an oriented cycle $\beta_1\cdots \beta_d$ starting at $v$, and the color of $\beta_j$ is $\alpha_j$, then
\[
[f_{\alpha_j}\cdots f_{\alpha_1}](v) = v.
\]
But then no power of $f_{\alpha_j}\cdots f_{\alpha_1}$ is zero, and so $\mathbb{V}$ is not nilpotent. Conversely, assume that $\Gamma_{\mathbb{V}}$ is acyclic. Then $\Gamma_{\mathbb{V}}$ contains finitely-many oriented paths, and there exists a natural number $N$ such that all paths in $\Gamma_{\mathbb{V}}$ have length strictly less than $N$. Consider a path $\alpha_1\cdots \alpha_n$ in $Q$ such that $f_{\alpha_n}\cdots f_{\alpha_1} \neq 0$. Then there are two basis elements $u$ and $v$ of $\mathbb{V}$ such that  
\[ 
f_{\alpha_n}\cdots f_{\alpha_1}(u) = v.
\] 
In turn, this implies the existence of an oriented path $\beta_1\cdots \beta_n$ in $\Gamma_{\mathbb{V}}$ from $u$ to $v$, such that the color of $\beta_j$ is $\alpha_j$ for all $j$. By the definition of $N$ this means $n < N$. Hence, finitely-many paths of $Q$ act via non-zero maps in $\mathbb{V}$, and $\mathbb{V}$ is nilpotent.

$(5)$: Recall from \cite[Section 4]{szczesny2011representations} that if $\mathbb{U} \subseteq \mathbb{V}\oplus \mathbb{W}$ is subrepresentation, then 
\[
\mathbb{U}=(\mathbb{U}\cap \mathbb{V}) \oplus (\mathbb{U} \cap \mathbb{W}). 
\]
Note that any decomposition $\mathbb{V} \cong \mathbb{V}_1\oplus \mathbb{V}_2$ induces a decomposition $V_u = A_u \oplus B_u$ for each $u \in Q_0$, such that \\
\[ 
f_{\alpha}(A_{s(\alpha)}) \subseteq A_{t(\alpha)} 
\]  
and 
\[ 
f_{\alpha}(B_{s(\alpha)}) \subseteq B_{t(\alpha)}, 
\] 
for all $\alpha \in Q_1$.
If we define $A = \bigcup_{u \in Q_0}{A_u\setminus\{0\}}$ and $B = \bigcup_{u \in Q_0}{B_u\setminus\{0\}}$, it follows that any arrow starting in $A$ (resp. $B$) must end in $A$ (resp. $B$). Hence, the partition $A \sqcup B$ of $(\Gamma_{\mathbb{V}})_0$ induces a separation of $\Gamma_\mathbb{V}$. Conversely, a separation of $\Gamma_\mathbb{V}$ induces a partition of each $V_u\setminus\{ 0\}$, which can be extended to a direct sum decomposition of each $V_u$ which is compatible with the action of the arrows. In other words, a separation of $\Gamma_\mathbb{V}$ induces a direct sum decomposition of $\mathbb{V}$, and the claim follows.
\end{proof}

\begin{mydef}
Let $c : \Gamma \rightarrow Q$ and $c' : \Gamma' \rightarrow Q$ be quiver maps (thought of as colored quivers as necessary).
\begin{enumerate}
	\item 
By a \emph{coefficient morphism} $c \rightarrow c'$, we mean a quiver map $\phi : \Gamma \rightarrow \Gamma'$ such that the following diagram commutes:

\begin{equation*}
\begin{tikzcd}[row sep=2em]
\Gamma \arrow[rr,"\phi"] \arrow[dr,swap,"c"]
&& \Gamma' \arrow[dl,,"c'"] \\
& Q
\end{tikzcd}.
\end{equation*}

\noindent We refer to this morphism by the map $\phi$. We say that $\phi$ is a \emph{coefficient isomorphism} if and only if it is bijective on vertices and arrows. In terms of colored quivers, a coefficient morphism maps $v$-colored vertices to $v$-colored vertices and $\alpha$-colored arrows to $\alpha$-colored arrows.
\item 
A collection of vertices $\mathcal{U}$ in $\Gamma$ is said to be \emph{successor-closed} if for every oriented path from $v$ to $u$ in $\Gamma$, $u \in \mathcal{U}$ implies $v \in \mathcal{U}$ as well. A full subquiver is said to be successor-closed if its vertex set is successor-closed. 
\item 
A subset $\mathcal{D}$ is said to be \emph{predecessor-closed} if for every oriented path from $d$ to $v$ in $\Gamma$, $d \in \mathcal{D}$ implies $v \in \mathcal{D}$ as well. A full subquiver is said to be predecessor-closed if its vertex set is predecessor-closed.
\end{enumerate}
\end{mydef} 

\begin{construction}\label{construction: fullsubquiver}
Let $\mathbb{V} = (V_u, f_{\alpha})$ and $\mathbb{W} = (W_u, g_{\alpha})$ be $\FF_1$-representations of $Q$, and $\phi : \mathbb{V} \rightarrow \mathbb{W}$ a morphism with component maps $\phi_u : V_u \rightarrow W_u$. Let $\mathcal{U}_\phi$ be the full subquiver of $\Gamma_\mathbb{V}$ with the vertex set $\{ x \in (\Gamma_{\mathbb{V}})_0 \mid \phi(x) \neq 0 \}$. Then $\phi$ induces a coefficient morphism
 \[
\Gamma_\phi:\mathcal{U}_\phi \to \Gamma_\mathbb{W}
\]
sending a $u$-colored vertex $x$ to $\phi(x) = \phi_u(x) \in \mathbb{W}_u$.\footnote{Note that $\Gamma_\phi:\Gamma_\mathbb{V} \to \Gamma_{\mathbb{W}}$ does not exist in general since the vertices of $\Gamma_{\mathbb{W}}$ are nonzero elements of $W_0$.} Indeed, suppose that there is an $\alpha$-colored arrow $x \xrightarrow[]{\alpha} y$ in $\Gamma_\mathbb{V}$, that is, $f_{\alpha}(x) = y$. Since $\phi$ is a morphism of representations, we have
\[
\phi_{t(\alpha)} f_{\alpha} = g_{\alpha} \phi_{s(\alpha)}.
\]
In particular, if neither $\phi_{s(\alpha)}(x)$ nor $\phi_{t(\alpha)}(y)$ are zero, then  
\[ 
g_{\alpha}(\phi_{s(\alpha)}(x)) = \phi_{t(\alpha)}(f_{\alpha}(x)) = \phi_{t(\alpha)}(y).
\] 
Hence, there is an arrow $\phi_{s(\alpha)}(x) \xrightarrow[]{\alpha} \phi_{t(\alpha)}(y)$ in $\Gamma_\mathbb{W}$ and $\Gamma_{\phi}$ is a coefficient morphism. 
\end{construction}

\noindent The purpose of the next lemma is to describe morphisms in $\Rep(Q,\FF_1)$ in terms of coefficient quivers. 

\begin{lem} \label{lemma: hom bijection}
For any two $\FF_1$-representations $\mathbb{V}= (V_u, f_{\alpha})$ and $\mathbb{W} = (W_u, g_{\alpha})$ of $Q$, there is a bijection between $\Hom(\mathbb{V},\mathbb{W})$ and the set of coefficient isomorphisms from successor-closed full subquivers of $\Gamma_\mathbb{V}$ to predecessor-closed full subquivers of $\Gamma_\mathbb{W}$.
\end{lem}  

\begin{proof} 
	Let $\phi : \mathbb{V} \rightarrow \mathbb{W}$ be a morphism of representations. If necessary, we will denote the component maps of $\phi$ by $\phi_u:V_u \to W_u$: however, if it is unnecessary to specify the source, we will simply denote $\phi_u(x)$ as $\phi(x)$. Let $\mathcal{U}_\phi$ be the quiver from Construction \ref{construction: fullsubquiver} and let $\mathcal{D}_\phi$ be its image under $\Gamma_\phi$. Consider the coefficient morphism
	\[
	\phi_{\bullet} :\mathcal{U}_\phi \rightarrow \mathcal{D}_\phi
	\]
	obtained by restricting $\Gamma_\phi$, which is surjective on vertices and arrows. Note that $\phi_\bullet$ is also injective on vertices by the $\FF_1$-linearity of $\phi$ and on arrows by Lemma \ref{l.basicprop}(1). It is clear from the definition of $\Gamma_\mathbb{V}$ that $\mathcal{U}_\phi$ is successor-closed, so we only need to show that $\mathcal{D}_\phi$ is predecessor-closed. First, we prove that $\mathcal{D}_\phi$ is a full subquiver of $\Gamma_{\mathbb{W}}$. Suppose that $\beta$ is an $\alpha$-colored arrow of $\Gamma_{\mathbb{W}}$ whose source and target are in $\mathcal{D}_\phi$. Since $\phi_{\bullet}$ is surjective on vertices, there exist vertices $x$ and $y$ of $\mathcal{U}_\phi$ such that $s(\beta) = \phi(x)$ and $t(\beta) = \phi(y)$. Then 
	\[
	\phi_{\bullet}(y) = \phi(y) = g_\alpha(\phi(x)) = \phi f_{\alpha}(x) = \phi_{\bullet}(f_{\alpha}(x)),
	 \]
	  where the last equality follows from the fact that $\phi(y) \neq 0$. Hence $y = f_{\alpha}(x)$, so that there is an $\alpha$-colored arrow $\gamma$ in $\mathcal{U}_{\phi}$ with $s(\gamma) = x$ and $t(\gamma) = y$. It now follows that $\beta = \phi_{\bullet}(\gamma)$ is an arrow in $\mathcal{D}_\phi$, so that $\mathcal{D}_\phi$ is full. Now suppose that $x \in \mathcal{U}_\phi$, so that $\phi(x) \in \mathcal{D}_\phi$ and $\phi(x) \neq 0$. Suppose that there is an oriented path $\beta_1\cdots \beta_d$ in $\Gamma_\mathbb{W}$ starting at $\phi(x)$ and ending at a vertex $z$ (so that in particular $z\neq 0$). If $\alpha_j$ denotes the color of $\beta_j$, then this means
	\[
	[g_{\alpha_d}\cdots g_{\alpha_1}]\phi(x) = z.
	\]
	Then $0 \neq z = \phi[f_{\alpha_d}\cdots f_{\alpha_1}](x)$ since $\phi$ is a morphism,
	so that $[f_{i_d}\cdots f_{i_1}](x) \in \mathcal{U}_\phi$ and hence $z \in \mathcal{D}_\phi$. It now follows that $\mathcal{D}_\phi$ is predecessor-closed. 
	
	Conversely, suppose that $\mathcal{U}$ is a successor-closed full subquiver of $\Gamma_\mathbb{V}$, $\mathcal{D}$ is a predecessor-closed full subquiver of $\Gamma_\mathbb{W}$, and that $\psi: \mathcal{U} \rightarrow \mathcal{D}$ is a coefficient isomorphism. For each $u \in Q_0$, define the map
	\[
	\psi^{\bullet}_u : V_u \rightarrow W_u
	\]
	as
	\[
	\psi^{\bullet}_u(x) = 
	\begin{cases}
	\psi(x), \textrm{ for all $x \in \mathcal{U}_0 \cap V_u$},\\
	0,  \textrm{ otherwise}.
	\end{cases}
	\]
	The injectivity of $\psi$ on vertices immediately implies the $\FF_1$-linearity of $\psi^{\bullet}_u$. We claim that $\psi^{\bullet} = (\psi^{\bullet}_u)_{u \in Q_0}$ is a morphism $\mathbb{V} \rightarrow \mathbb{W}$: this is equivalent to the claim that for each arrow $\alpha$ and each $x \in V_{s(\alpha)}$,
	\[
	\psi^{\bullet}_{t(\alpha)}(f_{\alpha}(x)) = g_{\alpha}(\psi^{\bullet}_{s(\alpha)}(x)).
	\]
	Note that since $\psi$ is a coefficient isomorphism, any $\alpha$-sink of $\mathcal{U}$ is mapped to an $\alpha$-sink of $\mathcal{D}$ of the same color. Since $\mathcal{D}$ is predecessor-closed, any $\alpha$-sink of $\mathcal{D}$ is an $\alpha$-sink of $\Gamma_\mathbb{W}$. If $x \in V_{s(\alpha)}$ is an $\alpha$-sink of $\mathcal{U}$, then either $f_{\alpha}(x) = 0$ or $f_{\alpha}(x)\not\in \mathcal{U}$. In either case, we have
	\[
	\psi^{\bullet}_{t(\alpha)}(f_{\alpha}(x)) = 0 = g_{\alpha}(\psi^{\bullet}_{s(\alpha)}(x)).
	\]
	Otherwise, $x$ is not an $\alpha$-sink of $\mathcal{U}$ and hence $f_{\alpha}(u) \in \mathcal{U}$. This means that $\psi^{\bullet}_{s(\alpha)}(x) = \psi(x)$ and $\psi^{\bullet}_{t(\alpha)}(f_{\alpha} (x)) = \psi(f_{\alpha} (x))$. Of course, there is an $\alpha$-colored arrow from $x$ to $f_{\alpha}(x)$ in $\Gamma_\mathbb{V}$, and hence in $\mathcal{U}$. Since $\psi$ is a coefficient isomorphism, there is an $\alpha$-colored arrow from $\psi(x)$ to $\psi(f_{\alpha}(x))$ in $\mathcal{D}$. This happens if and only if $g_{\alpha}(\psi(x)) = \psi(f_{\alpha}(x))$, which means $g_{\alpha}\psi^{\bullet}_{s(\alpha)}(x) = \psi^{\bullet}_{t(\alpha)}(f_{\alpha}(x))$. Hence $\psi^{\bullet}$ is a morphism of representations as claimed.  
	
	Finally, we note that $(\phi_{\bullet})^{\bullet} = \phi$ and $(\psi^{\bullet})_{\bullet} = \psi$ for all $\phi$ and $\psi$. Hence, these two constructions yield the desired bijection.
\end{proof} 

\noindent With the description of morphisms provided by Lemma \ref{lemma: hom bijection}, we can show that $\Gamma_{\mathbb{V}}$ is an isomorphism invariant of $\mathbb{V}$.

\begin{pro}\label{proposition: colored quiver}
Suppose that $c: \Gamma \rightarrow Q$ is a winding. Then there is a coefficient isomorphism $\phi: \Gamma \rightarrow \Gamma_\mathbb{V}$ for some $\FF_1$-representation $\mathbb{V}$ of $Q$, and $\mathbb{V}$ is well-defined up to isomorphism.
\end{pro} 

\begin{proof}  
Define the representation $\mathbb{V} = (V_u, f_{\alpha})$ as follows: $V_u$ consists of the $u$-colored vertices of $\Gamma$ plus an element $0$. If $x \in V_{s(\alpha)}$ is an $\alpha$-sink, then we define $f_{\alpha}(x) = 0$; otherwise there is a unique $\alpha$-colored arrow $\beta$ in $\Gamma$ with source $x$, in which case we define $f_{\alpha}(x) = t(\beta)$ (the target of $\beta$). Since $c$ is a winding, $f_{\alpha}$ is an $\FF_1$-linear map of $V_{s(\alpha)} \rightarrow V_{t(\alpha)}$. Note that $\Gamma$ and $\Gamma_{\mathbb{V}}$ have the same vertex set: the identity map on vertices then extends to a coefficient isomorphism $\phi : \Gamma \rightarrow \Gamma_\mathbb{V}$. If $\mathbb{W}$ is another $\FF_1$-representation of $Q$ with a coefficient isomorphism $\phi' : \Gamma \rightarrow \Gamma_\mathbb{W}$, then $f = \phi' \circ \phi^{-1}$ is a coefficient isomorphism $f: \Gamma_\mathbb{V} \rightarrow \Gamma_\mathbb{W}$. Taking $\mathcal{U}$ to be $\Gamma_\mathbb{V}$ and $\mathcal{D}$ to be $\Gamma_\mathbb{W}$ as in Lemma \ref{lemma: hom bijection}, we see that the associated map $f^{\bullet} : \mathbb{V} \rightarrow \mathbb{W}$ is bijective, and hence an isomorphism.
\end{proof} 

\noindent We end this section with the following straightforward result.  

\begin{lem}\label{lem: quotient rep}
Let $Q$ be a quiver, $\mathbb{V}$ and $\FF_1$-representation of $Q$, and $\mathbb{W}$ be a subrepresentation of $\mathbb{V}$.
\begin{enumerate}
\item 
$\Gamma_\mathbb{W}$ is the full subquiver of $\Gamma_\mathbb{V}$ obtained by removing vertices (and arrows with adjacent to those vertices) which do not correspond to $\mathbb{W}$. The map $c_{\mathbb{W}}$ is the restriction of $c_{\mathbb{V}}$ to this subquiver.
\item 
The coefficient quiver $\Gamma_{\mathbb{V}/\mathbb{W}}$ of the quotient $\mathbb{V}/\mathbb{W}$ is the full subquiver of $\Gamma_{\mathbb{V}}$ obtained by removing vertices (and arrows adjacent to those vertices) corresponding to $\mathbb{W}$. The map $c_{\mathbb{V/W}}$ is the restriction of $c_\mathbb{V}$ to this subquiver. 
\item 
Let 
\[ 
0 \rightarrow \mathbb{V} \xrightarrow[]{f} \mathbb{X} \xrightarrow[]{g} \mathbb{W} \rightarrow 0 
\] 
be a short exact sequence of $\FF_1$-representations of $Q$. Then $\Gamma_\mathbb{X}$ is obtained from the disjoint union $\Gamma_\mathbb{V} \sqcup \Gamma_\mathbb{W}$ by adding certain $\alpha$-colored arrows from $\alpha$-sinks of $\Gamma_\mathbb{W}$ to $\alpha$-sources of $\Gamma_\mathbb{V}$, for each $\alpha \in Q_1$. Under this decomposition, $\Gamma_\mathbb{V}$ is a predecessor-closed subquiver of $\Gamma_\mathbb{X}$ and $\Gamma_\mathbb{W}$ is a successor-closed subquiver of $\Gamma_\mathbb{X}$.
\end{enumerate}	
\end{lem}  

\begin{proof} 
(1): Identify $\mathbb{W}$ with the inclusion map $\iota : \mathbb{W} \rightarrow \mathbb{V}$. Then the map $\iota_{\bullet} : \mathcal{U}_{\iota} \rightarrow \mathcal{D}_{\iota}$ of Lemma \ref{lemma: hom bijection} identifies $\mathcal{U}_{\iota} = \Gamma_{\mathbb{W}}$ with the subquiver $\mathcal{D}_{\iota}$ of $\Gamma_{\mathbb{V}}$, from which the claim follows. 

(2): Consider the projection map $\pi : \mathbb{V} \rightarrow \mathbb{V}/\mathbb{W}$ with $\ker(\pi) = \mathbb{W}$. As in Lemma \ref{lemma: hom bijection} we have a coefficient isomorphism $\pi^{\bullet}: \mathcal{U}_{\pi} \rightarrow \mathcal{D}_{\pi}$. Note that the vertices of $\mathcal{U}_{\pi}$ are precisely the non-zero elements of $\mathbb{V}$ which do not lie in $\mathbb{W}$, and that $\mathcal{D}_{\pi} = \Gamma_{\mathbb{V}/\mathbb{W}}$ since each non-zero element of $\mathbb{V}/\mathbb{W}$ lies in the image of $\pi$. The claim now follows. 

(3): Let $\mathbb{X} = (X_u,f_\alpha)$. Since $\mathbb{V} = \operatorname{Im}(f)$ is a subrepresentation of $\mathbb{X}$, (1) implies that $\Gamma_\mathbb{V}$ is a full subquiver of $\Gamma_\mathbb{X}$. It is predecessor-closed since $x \in \mathbb{V}$ implies $f_{\alpha}(x) \in \mathbb{V}$ for all $\alpha \in Q_1$. Then $\mathbb{X}/\mathbb{V} \cong \mathbb{W}$, so (2) implies that $\Gamma_\mathbb{W}$ is a full subquiver of $\Gamma_\mathbb{X}$ whose vertex set is $(\Gamma_\mathbb{X})_0\setminus (\Gamma_\mathbb{V})_0$. Since $\Gamma_\mathbb{V}$ is predecessor-closed and $(\Gamma_\mathbb{X})_0 = (\Gamma_\mathbb{V})_0 \sqcup (\Gamma_\mathbb{W})_0$, it follows that $\Gamma_\mathbb{W}$ is successor-closed. Hence, any $\alpha$-colored arrow $\beta$ of $\Gamma_\mathbb{X}$ which does not lie in $\Gamma_\mathbb{V}\sqcup\Gamma_\mathbb{W}$ must start in $\Gamma_\mathbb{W}$ and terminate in $\Gamma_\mathbb{V}$. The fact that the map $c_\mathbb{X} : \Gamma_\mathbb{X} \rightarrow Q$ is a winding then implies that $\beta$ must start at an $\alpha$-sink of $\Gamma_\mathbb{W}$ and terminate at an $\alpha$-source of $\Gamma_\mathbb{V}$.
\end{proof} 

%%%%%%%%%%%%%%n-LOOP QUIVER STUFF (NEW)%%%%%%%%%%%%%%%%%%%% 

\section{Representations of $n$-Loop Quivers over $\FF_1$}\label{section: n-loop representations}

In this section we apply the ideas of Section \ref{section: coefficient quivers} to $\wild_n$. Recall that $\wild_n$ is the quiver with one vertex and $n$ arrows $\alpha_1,\ldots , \alpha_n$, where we treat $\{\alpha_1,\ldots , \alpha_n\}$ as a totally ordered set in the obvious way. If $n=1$ this is known as the \emph{Jordan quiver}. An $\FF_1$-representation of $\wild_n$ is an $\FF_1$-vector space $V_0$ along with an ordered $n$-tuple $(f_1,\ldots ,f_n)$ of $\FF_1$-linear maps $f_i : V_0 \rightarrow V_0$, where $f_i$ corresponds to $\alpha_i$. If this representation is nilpotent, then the $f_i$ are all nilpotent (as are their powers and products). One can easily see that two representations $\mathbb{V} = (V_0,f_1,\ldots , f_n)$ and $\mathbb{W} = (W_0,g_1,\ldots ,g_n)$ of $\wild_n$ are isomorphic if and only if there is an $\FF_1$-linear isomorphism $\varphi : V_0 \rightarrow W_0$ such that $g_i = \varphi f_i\varphi^{-1}$ for each $i = 1,\ldots , n$. 

For any $1 \le i \le n$, there is a functor  
\[
D_i : \Rep(\wild_n,\FF_1)_{\nil} \rightarrow \Rep(\wild_{n-1},\FF_1)_{\nil}  
\]
defined by deleting the $i^{th}$ arrow and relabeling the arrows according to the total ordering on the remainder. 

\begin{construction}\label{construction: GammaFM} 
Let $n\geq 2$ be a natural number. If $M$ is a nilpotent $\FF_1$-representation of $\wild_n$, then we construct a new nilpotent $\FF_1$-representation of $\wild_{n-1}$, denoted $F(M)$. In light of Proposition \ref{proposition: colored quiver}, it suffices to construct its coefficient quiver $\Gamma_{F(M)}$. The coefficient quiver $\Gamma_{F(M)}$ will be the disjoint union of $\Gamma_{D_{n-1}(M)} \sqcup \Gamma_{D_n(M)}$ with the following extra arrows: for each maximal $\alpha_1$-colored path in $\Gamma_M$ with source $u$ and target $v$, draw an $\alpha_1$-colored arrow from the copy of $v$ in $\Gamma_{D_{n-1}(M)}$ to the copy of $u$ in $\Gamma_{D_n(M)}$. The map $c_{F(M)}:\Gamma_{F(M)} \rightarrow \wild_{n-1}$ is given as follows: the restrictions to the full subquivers $\Gamma_{D_{n-1}(M)}$ and $\Gamma_{D_n(M)}$ are equal to the maps $c_{D_{n-1}(M)}$ and $c_{D_n(M)}$, respectively. The additional $\alpha_1$-colored arrows are then mapped to $\alpha_1$.

 It is useful to think of $\Gamma_{D_{n-1}(M)}$ as a successor-closed full subquiver of $\Gamma_{F(M)}$ and $\Gamma_{D_n(M)}$ a predecessor-closed full subquiver lying beneath it. Note that $\Gamma_{D_{n-1}(M)}$ and $\Gamma_{D_n(M)}$ have the same vertex set, that $\alpha_n$-colored arrows in $\Gamma_M$ become $\alpha_{n-1}$-colored in $\Gamma_{D_{n-1}(M)}$, and that the other arrows in $\Gamma_M$ retain the same colors in $\Gamma_{D_{n-1}(M)}$ and $\Gamma_{D_n(M)}$.
\end{construction}

 \begin{myeg}\label{example: M to F(M)}
Let $M$ be an object of $\Rep(\wild_3,\FF_1)_{\nil}$ with the following associated coefficient quiver:
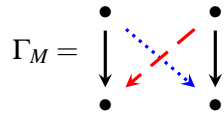
\[ \Gamma_M =
\begin{tikzcd}[arrow style=tikz,>=stealth,row sep=2em]
\bullet \arrow[d,very thick] \arrow[dr,very thick,blue,dotted] & \bullet \arrow[d,very thick] \arrow[dl,very thick,red, dashed] \\  
\bullet & \bullet
\end{tikzcd}
\]
where $\alpha_1$-colored arrows are black , $\alpha_2$-colored arrows are dotted blue, and $\alpha_3$-colored arrows are dashed red. Then $F(M)$ has the following coefficient quiver:
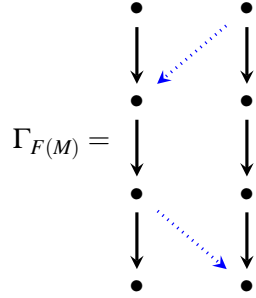
\[ 
\Gamma_{F(M)} = 
\begin{tikzcd}[arrow style=tikz,>=stealth,row sep=2em]
\bullet \arrow[d,very thick] & \bullet \arrow[d,very thick] \arrow[dl,blue,dotted,very thick] \\  
\bullet \arrow[d,very thick] & \bullet \arrow[d,very thick] \\ 
\bullet \arrow[d,very thick] \arrow[dr,blue,dotted,very thick] & \bullet \arrow[d,very thick]  \\ 
\bullet & \bullet
\end{tikzcd}
\]  
The top four vertices correspond to $D_2(M)$ and the bottom four correspond to $D_3(M)$.
\end{myeg}

Note that in Example \ref{example: M to F(M)}, $F(M)$ fits into a non-split short exact sequence 
\begin{equation} \label{eq: exact seq gluing}
0 \rightarrow D_3(M) \rightarrow F(M) \rightarrow D_2(M) \rightarrow 0.\footnote{One can easily see that this is an exact sequence by considering the corresponding coefficient quivers from Lemma \ref{lem: quotient rep}.}
\end{equation}
This is not a coincidence: by the construction of $\Gamma_{F(M)}$, there is an analogous non-split short exact sequence for \emph{any} $M$. Note that the function $\Gamma_{M} \rightarrow \Gamma_{F(M)}$ preserves connectedness and acyclicity. This means that the function $M\mapsto F(M)$ preserves  indecomposability and nilpotency. \par \medskip

There is a canonical way to partition the $\alpha_{n-1}$-colored arrows of $\Gamma_{F(M)}$ which shall prove useful below. To begin, arrange the vertices of $\Gamma_{F(M)}$ in such a way that all the $\alpha_1$-colored arrows point downwards, and such that the maximal $\alpha_1$-colored paths are arranged left-to-right from longest to shortest. By Construction \ref{construction: GammaFM}, each of these paths has an even number of vertices: suppose that the $i^{th}$ maximal $\alpha_1$-colored path in this embedding contains $2\lambda_i$-vertices. Now take the full subquiver of $\Gamma_{F(M)}$ whose vertex set is formed from the first $\lambda_i$ vertices in the $i^{th}$ maximal path, for each $i$ (starting with the $\alpha_1$-source and moving along the path). Call this full subquiver $\mathcal{U}_{F(M)}$, and call the full subquiver on the remaining vertices $\mathcal{D}_{F(M)}$. Note that $\mathcal{U}_{F(M)} \cong \Gamma_{D_{n-1}(M)}$ and $\mathcal{D}_{F(M)} \cong \Gamma_{D_n(M)}$ by the construction of $\Gamma_{F(M)}$. Hence, all the $\alpha_{n-1}$-colored arrows of $\Gamma_{F(M)}$ lie either in $\mathcal{U}_{F(M)}$ or $\mathcal{D}_{F(M)}$. In fact, the $\alpha_{n-1}$-colored arrows of $\mathcal{U}_{F(M)}$ correspond to the $\alpha_n$-colored arrows of $\Gamma_M$ and the $\alpha_{n-1}$-colored arrows of $\mathcal{D}_{F(M)}$ correspond to the $\alpha_{n-1}$-colored arrows of $\Gamma_M$. Of course, there is a bijection with the $\alpha_i$-colored arrows of $\mathcal{U}_{F(M)}$ (resp. $\mathcal{D}_{F(M)}$) and those of $\Gamma_M$ for each $1 \le i < n-1$. Hence, $M$ can be recovered from $\Gamma_{F(M)}$. 

 Since $\Gamma_M$ is an isomorphism invariant of $M$ (up to coefficient isomorphism), $M \cong N$ implies $\Gamma_{F(M)} \cong \Gamma_{F(N)}$ (again via a coefficient isomorphism) and hence $F(M) \cong F(N)$. Conversely, suppose that $F(M) \cong F(N)$. We claim that this implies $M \cong N$. Indeed, let $M = (M_0,f_1,f_2,\ldots , f_n)$ and $N = (N_0,g_1,g_2,\ldots , f_n)$. An isomorphism $\phi : F(M) \rightarrow F(N)$ induces a coefficient isomorphism $ \Gamma_{F(M)} \rightarrow \Gamma_{F(N)}$ which we also call $\phi$. But then restriction of $\phi$ induces coefficient isomorphisms $\mathcal{U}_{F(M)} \rightarrow \mathcal{U}_{F(N)}$ and $\mathcal{D}_{F(M)}\rightarrow \mathcal{D}_{F(N)}$. The first restriction implies that $\phi$ induces an $\FF_1$-linear map $\psi : M_0 \rightarrow N_0$ such that $g_i = \psi f_i \psi^{-1}$ for $1 \le i < n-1$ and $g_n = \psi f_n \psi^{-1}$. Now note that the second restriction $\mathcal{D}_{F(M)} \rightarrow \mathcal{D}_{F(N)}$ induces the \emph{same} map $M_0 \rightarrow N_0$. Indeed, any maximal $\alpha_1$-colored path in $\mathcal{U}_{F(M)}$ is connected to exactly one maximal $\alpha_1$-colored path in $\mathcal{D}_{F(M)}$ by an $\alpha_1$-colored arrow in $\Gamma_{F(M)}$. Recall the canonical identifications 
\[  
\mathcal{U}_{F(M)} \cong \Gamma_{D_{n-1}(M)},\quad \mathcal{D}_{F(M)} \cong \Gamma_{D_n(M)} 
\]  
described above, along with the fact that each of these coefficient quivers has the same vertex set as $\Gamma_M$. Under these identifications, the vertices in the maximal $\alpha_1$-colored path of $\mathcal{U}_{F(M)}$ and those in the corresponding maximal $\alpha_1$-colored path of $\mathcal{D}_{F(M)}$ get sent to the same vertices of $\Gamma_M$. Hence, the global map $\phi$ restricts to the same vertex map on $\mathcal{U}_{F(M)}$ and $\mathcal{D}_{F(M)}$, as we claimed. It now follows from the second restriction that $\psi$ must also satisfy $g_{n-1} = \psi f_{n-1} \psi^{-1}$. Finally, this implies $M \cong N$, as we wished to show. In terms of growth functions of indecomposables of quivers (Definition \ref{definition: indecomposables growth}), this implies the following theorem:

\begin{mytheorem}\label{theorem: n to n-1}  
The following inequalities hold for all natural numbers $d$ and $n\geq 2$: 
\[
\NI_{\wild_{n-1}}(d) \le \NI_{\wild_n}(d) \le \NI_{\wild_{n-1}}(2d). 
\]
\end{mytheorem}  
\begin{proof}  
We have shown that the assignment $M \mapsto F(M)$ induces an injective map from the isomorphism classes of indecomposable $\wild_n$-representations to the isomorphism classes of indecomposable $\wild_{n-1}$-representations. Furthermore, $\dim F(M) = 2\cdot \dim (M)$ for every such $M$. This shows that
\[
\NI_{\wild_n}(d) \le \NI_{\wild_{n-1}}(2d), \quad \forall d \in \mathbb{N}. 
\]
The other inequality $\NI_{\wild_{n-1}}(d) \le \NI_{\wild_n}(d)$ follows from the fact that $\wild_{n-1}$ is a subquiver of $\wild_n$, so that $\Rep(\wild_{n-1},\FF_1)_{\nil}$ is a full subcategory of $\Rep(\wild_n,\FF_1)_{\nil}$.
\end{proof}

\begin{cor}\label{corollary: L_m and L_n}
Let $2 \le m < n$ be natural numbers. Then for all $d \in \NN$, 
\[ 
\NI_{\wild_m}(d) \le \NI_{\wild_n}(d) \le \NI_{\wild_m}(2^{n-m}d). 
\]
%For any two natural numbers $m$ and $n$, there is a faithful functor $F: \Rep(\wild_m,\FF_1)_{\nil} \rightarrow \Rep(\wild_n,\FF_1)_{\nil}$ which preserves indecomposability, such that $\dim_{\FF_1}(F(V))/\dim_{\FF_1}(V)$ is a positive constant for all non-zero representations $V$.
\end{cor} 

\begin{proof} 
This follows from repeated application of Theorem \ref{theorem: n to n-1}.
\end{proof}   

Corollary \ref{corollary: L_m and L_n} suggests an order relation on quivers based on the growth of their representation functions. For any natural number $C$, let $\mu_C : \NN \rightarrow \NN$ denote multiplication by $C$. For two quivers $Q$ and $Q'$, we write $Q \le_{\nil} Q'$ if there exists a natural number $C$ such that $\NI_Q = O(\NI_{Q'} \circ \mu_C)$ in big-$O$ notation. More explicitly, $Q \le_{\nil}Q'$ if and only if there exists a positive real number $D$ and a natural number $C$ such that  
\[
\NI_Q(n) \le D\NI_{Q'}(Cn), \text{ for all sufficiently large $n$.}
\]
It is straightforward to check that $\le_{\nil}$ is a reflexive and transitive relation. If we define $Q \approx_{\nil} Q'$ if and only if $Q\le_{\nil}Q'$ and $Q'\le_{\nil}Q$, then $\approx_{\nil}$ becomes an equivalence relation on quivers.

\begin{myeg} 
The following statements demonstrate the utility of the notion of $\approx_{\nil}$: 
\begin{enumerate} 
\item It will follow from Theorem \ref{theorem: finite type} that $Q$ has finitely many indecomposables (up to isomorphism) if and only if $Q\approx_{\nil} T$, where $T$ is any tree quiver.
\item If $C_n$ is an equioriented $\tilde{\mathbb{A}}_n$ quiver, then it follows from  \cite[Section 11]{szczesny2011representations} that $\wild_1 \approx_{\nil} C_n$ for all $n \ge 2$.
\item For all $n \geq 2$, $\wild_1 \not\approx_{\nil} \wild_n$. Indeed, $\NI_{\wild_1} \circ \mu_C \in O(1)$ for all $C$, but $\limsup_m\NI_{\wild_n}(m) = \infty$ for $n\geq 2$.
\item By Corollary \ref{corollary: L_m and L_n}, $\wild_m \approx_{\nil} \wild_n$ whenever $m$ and $n$ are both at least $2$.
\end{enumerate}
\end{myeg}  

\noindent The result below demonstrates that $\wild_2$ is an ``upper bound'' for quivers with respect to the $\le_{\nil}$ relation.

\begin{mytheorem}\label{theorem: upper bound}
Let $Q$ be a quiver. Then $Q \le_{\nil} \wild_2$.
\end{mytheorem} 
\begin{proof} 
In light of Corollary \ref{corollary: L_m and L_n}, it suffices to show that $Q \le_{\nil} \wild_n$ for some finite $n \geq 1$. First set $n = |Q_0| + |Q_1|$, and put a total order on the set $Q_0\sqcup Q_1$ such that each vertex is less than each arrow. The total order allows us to identify $Q_0\sqcup Q_1$ with the arrow set of $\wild_n$. Now let $M$ be an indecomposable nilpotent representation of $Q$, and send $M \mapsto \Gamma_M$. We construct a new $\FF_1$-representation of $\wild_n$ called $G(M)$ by specifying its coefficient quiver. The quiver $\Gamma_{G(M)}$ is obtained by adding the following vertices and arrows to $\Gamma_M$: to each $v$-colored vertex $x$ of $\Gamma_M$, add a new vertex $s_x$ and a single $v$-colored arrow $\alpha_x$ such that $t(\alpha_x) = x$ and $s(\alpha_x) = s_x$. Note that the arrows of the resulting quiver $\Gamma_{G(M)}$ are colored by the elements of $Q_0 \sqcup Q_1$. Now recolor each vertex in $\Gamma_{G(M)}$ with the unique vertex of $\wild_n$. The map $c_{G(M)} : \Gamma_{G(M)} \rightarrow \wild_n$ is given by mapping each vertex and arrow to its associated color.  

 By construction, $\dim_{\FF_1}G(M)  = 2\dim_{\FF_1}(M)$. Furthermore, $G(M)$ is indecomposable. To see this, recall that $\Gamma_M$ is connected since $M$ is indecomposable. But $\Gamma_{G(M)}$ is obtained from $\Gamma_M$ by adding new arrows, and hence is connected as well. By Lemma \ref{l.basicprop}(5), $G(M)$ is indecomposable. Finally,  $G(M) \cong G(N)$ if and only if $M \cong N$. Indeed, any coefficient isomorphism $G(M)\rightarrow G(N)$ must map the $Q_0$-colored arrows of $G(M)$ to the $Q_0$-colored arrows of $G(N)$. Furthermore, the source of each $Q_0$-colored arrow has indegree $0$ and outdegree $1$. Hence, any such coefficient isomorphism must restrict to a coefficient isomorphism between the subquivers of $G(M)$ and $G(N)$ obtained by deleting the sources of the $Q_0$-colored arrows. The claim now follows from the simple observation that $\Gamma_M$ may be recovered from $\Gamma_{G(M)}$ by deleting the sources of $Q_0$-colored arrows. The assignment $M \mapsto G(M)$ implies the inequality $\NI_Q(m) \le \NI_{\wild_n}(2m)$ for all $m \geq 1$, which in turn implies $Q \le_{\nil} \wild_n$.
\end{proof}   

\begin{myeg} 
We illustrate the construction of Theorem \ref{theorem: upper bound} with an example. Let $Q$ be the Kronecker quiver, with vertex set $\{1,2\}$ and two arrows $1\xrightarrow[]{\alpha} 2$ and $1\xrightarrow[]{\beta} 2$. Let $M$ be the indecomposable representation with coefficient quiver
\[ \Gamma_M= 
\begin{tikzcd}[arrow style=tikz,>=stealth,row sep=2em] 
1 \arrow[r,"\alpha"] & 2 &  \arrow[l,swap,"\beta"]  1\\ 
\end{tikzcd}. 
\] 
Then $G(M)$ is the indecomposable $\wild_4$-representation with coefficient quiver 
\[ 
\Gamma_{G(M)} = 
\begin{tikzcd}[arrow style=tikz,>=stealth,row sep=2em]
\bullet \arrow[r,"\alpha"] & \bullet & \arrow[l,swap,"\beta"] \bullet \\  
\arrow[u, "1"] \bullet & \arrow[u,"2"] \bullet & \arrow[u,"1"] \bullet \\
\end{tikzcd}.
\]
\end{myeg} 

\begin{rmk} 
For an algebraically closed field $k$ and a wild quiver $Q$, the category of $k$-representations of $Q$ admits a fully-faithful embedding into the category of $k$-representations of $\wild_2$. It is not currently known whether a similar statement holds over $\FF_1$.
\end{rmk}

%%%%%%%%%%%%%%%%(FINITE TYPE)%%%%%%%%%%%%%%%%%% 

\section{Growth of indecomposable quiver representations over $\FF_1$} \label{section: growh of indecomposable quiver reps}

In this section, we study quivers with respect to the $\approx_{\nil}$ relation. We characterize quivers of finite and bounded representation type, showing that these are precisely the tree and quivers of type $\tilde{\mathbb{A}}_n$. We then characterize unbounded quivers in terms of certain categorical embeddings. %We then show how bounded representation type partitions connected quivers in a way similar to the ``tame-wild dichotomy'' in finite-dimensional algebras.

\subsection{Quivers of finite representation type over $\mathbb{F}_1$}

Recall that $Q$ has finite representation type over $\mathbb{F}_1$ if and only if $\Rep(Q,\FF_1)_{\nil}$ has finitely many isomorphism classes of indecomposables. In terms of the $\le_{\nil}$ relation, $Q$ is of finite type if and only if $Q \approx_{\nil} \wild_0$, the quiver with one vertex and zero arrows. In \cite{szczesny2011representations}, Szczesny proves that trees are of finite representation type over $\FF_1$. In this section, we show that the converse holds for connected quivers. To begin, consider any subquiver $S$ of $Q$. Then we have a functor
\[
\mathcal{F}: \Rep (S,\FF_1) \rightarrow \Rep (Q,\FF_1)
\]
which takes an $S$-representation $M$ to the $Q$-representation $\mathcal{F}(M)$ which restricts to $M$ on $S$, and which satisfies $\mathcal{F}(M)_v = 0$ whenever $v \not\in S_0$ and $\mathcal{F}(M)_{\alpha} = 0$ whenever $\alpha \not\in S_1$. If $M$ is nilpotent then so is $\mathcal{F}(M)$, so the functor $\mathcal{F}$ restricts to a functor $\Rep (S,\FF_1)_{\nil} \rightarrow \Rep(Q,\FF_1)_{\nil}$. In the other direction, there is a restriction functor  
\[
\Res_S : \Rep(Q,\FF_1) \rightarrow \Rep(S,\FF_1) 
\]
 which restricts to a functor $\Rep(Q,\FF_1)_{\nil} \rightarrow \Rep(S,\FF_1)_{\nil}$. %THESE PROBABLY FORM AN ADJOINT PAIR}

\begin{lem}\label{lemma: l.subq}
Let $S$ be a subquiver of $Q$. Then the functor $\mathcal{F}: \Rep (S,\FF_1) \to \Rep (Q,\FF_1)$ is full and preserves indecomposability of objects. An equivalent statement holds for nilpotent representations.
\end{lem} 
\begin{proof} 
Let $M,N \in \Rep (S,\FF_1)$ and $\phi: \mathcal{F}(M) \rightarrow \mathcal{F}(N)$ be a morphism in $\Rep(Q,\FF_1)$. Define a morphism $\tilde{\phi}: M\rightarrow N$ in $\Rep(S,\FF_1)$ as follows: $\tilde{\phi}_u = \phi_u$ if $u \in S_0$ and $\tilde{\phi}_u = 0$ otherwise. One readily checks that $\tilde{\phi}$ is a morphism in $\Rep(S,\FF_1)$, so it only remains to check that $\mathcal{F}(\tilde{\phi}) = \phi$. It is clear that $\mathcal{F}(\tilde{\phi})_v = \phi_v$ if $v \in S_0$, so assume that $v$ is a vertex of $Q$ but not of $S$. Then $\mathcal{F}(M)_v = \mathcal{F}(N)_v = 0$ and so $\phi_v = 0$, from which it follows that $\mathcal{F}(\tilde{\phi})_v = \phi_v$ for all $v \in Q_0$. It is now clear that $\mathcal{F}$ preserves indecomposability: if $\mathcal{F}(M) \cong N_1 \oplus N_2$ in $\Rep(Q,\FF_1)$, then $N_1$ and $N_2$ are also $S$-representations and we can find a decomposition $M = M_1 \oplus M_2$ in $\Rep(S,\FF_1)$ which preserves the dimension vectors of $N_1$ and $N_2$. The proof for nilpotent representations is identical. 
\end{proof} 

\begin{lem}\label{lemma: l.cyc}
Let $Q$ be a quiver of type $\tilde{\mathbb{A}}_m$. Then $Q$ is not of finite type over $\FF_1$.
\end{lem} 

\begin{proof} 
Let $m = |Q_0|$, and order the vertices of $Q$ as $1$, \ldots , $m$. Up to a reordering of the vertices, we may assume that the arrow $\alpha_1$ between $1$ and $2$ starts at $1$. Consider the family of representations $M(n)$, where $n\geq 1$ is a natural number: $M(n)_i = [n]$ for $i = 1,\ldots , m$; $M(n)_{\alpha_1}(k)=k-1$; and $M(n)_{\alpha_i} = \id_{[n]}$ for $i > 1$. We show that the $M(n)$'s are an infinite set of pairwise non-isomorphic indecomposable representations. Indeed, suppose that we had an isomorphism $\phi: M(n) \rightarrow X\oplus Y$. Write $\chi_i\oplus \psi_i$ for the map between $X_i\oplus Y_i$ and $X_{i+1}\oplus Y_{i+1}$ (where $i$ is read mod $m$). For each $1 \le i \le m$, let $A_i = \phi_i^{-1}(X_i)$ and $B_i = \phi_i^{-1}(Y_i)$. Then $[n] = A_i \sqcup B_i$ for each $i$. For $i > 1$, the commutativity of the square at $\alpha_i$ implies that $A_1 = A_2 = \ldots = A_m$ and $B_1 = B_2 = \ldots = B_m$. Hence, we write $A = A_1$ and $B = B_1$. The commutativity of the square at $\alpha_1$ implies that $M(n)_{\alpha_1}$ leaves $A$ and $B$ invariant. If $n \in A$, this immediately implies that $X = M(n)$ and $Y = 0$. In other words, $M(n)$ is indecomposable, from which the claim follows.
\end{proof} 

\begin{mytheorem} \label{theorem: finite type}
Let $Q$ be a quiver. Then $Q$ is of finite type over $\FF_1$ if and only if $Q$ is a tree.
\end{mytheorem} 

\begin{proof} 
It only remains to prove that if $Q$ is not a tree, then $Q$ is not of finite type. If $Q$ is not a tree, then the underlying graph of $Q$ contains a cycle. Let $C$ be the subquiver of $Q$ corresponding to this cycle. By Lemma \ref{lemma: l.cyc}, $C$ is not of finite type. By Lemma \ref{lemma: l.subq}, this means that $\Rep(Q,\FF_1)_{\nil}$ contains infinitely many non-isomorphic indecomposables, and hence is not of finite type.
\end{proof}   

%%%%%%%%%%%%%%%(Bounded Type)%%%%%%%%%%%%%%

\subsection{Quivers of bounded representation type over $\FF_1$}

We say that $Q$ has \emph{bounded representation type over $\FF_1$} if $Q \le_{\nil} \wild_1$. In other words, $Q$ is of bounded representation type if and only if $\NI_Q = O(1)$. From the previous subsection, tree quivers are of bounded representation type over $\FF_1$. In this subsection, we prove that if $Q$ is of bounded representation type and $Q$ is not a tree, then $Q$ is of type $\tilde{\mathbb{A}}_n$.

Recall that an undirected graph $G$ can be viewed as an $1$-dimensional simplicial complex, and its first homology $H_1(G,\ZZ_2)$ is called its \emph{cycle space} \cite{AlgGraphTheory}. It can be interpreted as the set of all (spanning) Eulerian subgraphs of $G$, with addition being given by symmetric difference. As a $\ZZ_2$-vector space it can be given a \emph{fundamental cycle basis} as follows: take a spanning tree $T$ of $G$. For each edge $e$ not in $T$, let $C_e$ be the cycle consisting of $e$ together with the path in $T$ connecting its endpoints. The collection $\{ C_e\mid e\not\in T\}$ forms a $\ZZ_2$-basis for $H_1(G,\ZZ_2)$ with the property that each element contains an edge not contained in the others \cite{LiebchenRizzi2007}.

\begin{pro}\label{proposition: rank}
Let $Q$ be a connected quiver with underlying graph $\overline{Q}$. If $\dim_{\ZZ_2}H_1(\overline{Q},\ZZ_2) \geq 2$, then there is an exact, fully faithful functor  
\[ 
\mathcal{F}: \Rep(\wild_2,\FF_1)_{\nil} \rightarrow \Rep(Q,\FF_1)_{\nil}.
\] 
Furthermore, $\mathcal{F}$ preserves indecomposability, and there exists a natural number $C$ such that  
\[  
\dim_{\FF_1}(\mathcal{F}(M)) = C\cdot\dim_{\FF_1}(M), 
\]
 for all objects $M \in \Rep(\wild_2,\FF_1)_{\nil}$. In particular, $Q \approx_{\nil} \wild_2$.   
\end{pro} 
\begin{proof} 
Let $T$ be a spanning tree of $\overline{Q}$, with fundamental cycle basis $\{ C_e\mid e \not\in T\}$. Since  
\[
\dim_{\ZZ_2}H_1(\overline{Q},\ZZ_2) \geq 2, 
\] 
this basis contains at least two elements $C_{\alpha}$ and $C_{\beta}$. In other words, $C_{\alpha}$ and $C_{\beta}$ are two fundamental cycles in $\overline{Q}$ with the property that $C_{\alpha}$ contains $\alpha$ but not $\beta$, and $C_{\beta}$ contains $\beta$ but not $\alpha$. Since $\overline{Q}$ is connected, we can find a path $w$ (possibly of length $0$) that connects a vertex in $C_{\alpha}$ to a vertex in $C_{\beta}$. By possibly choosing a shorter path, we may assume that $w$ passes through neither $\alpha$ nor $\beta$. Let $S$ denote the subquiver of $Q$ whose arrows are precisely those corresponding to $C_{\alpha}$, $C_{\beta}$ and $w$. By abuse of notation, we use $\alpha$, $C_{\alpha}$ etc. to refer to the arrows/subquivers in $Q$ corresponding to the edges/subgraphs in $\overline{Q}$. Note that $S$ is a connected subquiver. 
We now define a map on objects $\mathcal{F} : \Rep(\wild_2,\FF_1)_{\nil} \rightarrow \Rep(S,\FF_1)_{\nil}$. If $\mathbb{V} = (V,f_1,f_2)$ is a nilpotent representation of $\wild_2$, then $\mathcal{F}(\mathbb{V})$ is defined by the following three criteria:  
\begin{enumerate}
\item 
At each vertex $v \in S_0$, $\mathcal{F}(\mathbb{V})_v = V$, and zero otherwise; 
\item 
$\mathcal{F}(f_\alpha) = f_1$ and $\mathcal{F}(f_\beta) = f_2$; 
\item 
$\mathcal{F}(f_\gamma) = \id_{V}$ for all $\gamma \in S_1\setminus\{ \alpha,\beta\}$, and zero otherwise.  
\end{enumerate}
We claim that $\mathcal{F}$ is a functor. Indeed, let $\phi : \mathbb{V}=(V,f_1,f_2) \rightarrow \mathbb{W}=(W,g_1,g_2)$ be a morphism in $\Rep(\wild_2,\FF_1)_{\nil}$. Setting $\mathcal{F}(\phi)_v = \phi$ for each vertex $v \in S_0$ yields a morphism $\mathcal{F}(\phi) : \mathcal{F}(\mathbb{V}) \rightarrow \mathcal{F}(\mathbb{W})$, and it is easy to check that $\mathcal{F}$ respects identity morphisms and composition. It is clear that $\mathcal{F}$ is a faithful, exact functor that sends an $n$-dimensional representation to an $|S_0|n$-dimensional representation. We claim that it is also full, and preserves indecomposability.  

To prove fullness, let $\psi : \mathcal{F}(\mathbb{V}) \rightarrow \mathcal{F}(\mathbb{W})$ be a morphism in $\Rep(S,\FF_1)_{\nil}$. Commutativity at $\alpha$ means that $\psi_{t(\alpha)}f_1 = g_1\psi_{s(\alpha)}$, whereas commutativity at $\beta$ means that $\psi_{t(\beta)}f_2 = g_2\psi_{s(\alpha)}$. Since $C_{\alpha}$ does not contain $\beta$, commutativity at the other arrows of $C_{\alpha}$ (which are identity maps) implies that $\psi_{u} = \psi_v$ for all vertices $u$ and $v$ of $C_{\alpha}$. Since $w$ connects to $C_{\alpha}$ and contains neither $\alpha$ nor $\beta$, $\psi_u = \psi_v$ for all vertices $u, v$ in $w$ or $C_{\alpha}$. Since the other endpoint of $w$ lies in $C_{\beta}$, and $\beta$ is the only non-identity arrow in $C_{\beta}$, we can finally conclude that $\psi_u = \psi_v$ for all $u,v \in S_0$. In other words, $\psi = \mathcal{F}(\psi_u)$ for any vertex $u$ of $S_0$ and $\mathcal{F}$ is full. Since $\mathcal{F}$ is fully faithful, we have monoid isomorphisms $\operatorname{End}(\mathbb{V}) \cong \operatorname{End}(\mathcal{F}(\mathbb{V}))$, for every object $\mathbb{V}$. Lemma 4.2 of \cite{szczesny2011representations} now implies that $\mathcal{F}$ preserves indecomposables.
 
%Now we prove that $\mathcal{F}$ preserves indecomposability. Suppose that $\mathbb{V}=(V,f_1,f_2)$ is an indecomposable representation of $\wild_2$ over $\mathbb{F}_1$, and that $\mathcal{F}(\mathbb{V}) = X \oplus Y$ for some $S$-representations $X$ and $Y$. Then for every arrow $\gamma \in S_1$, $X_{s(\gamma)}$ (resp. $Y_{s(\gamma)}$) gets sent to $X_{t(\gamma)}$ (resp. $Y_{t(\gamma)}$) by $\mathcal{F}(\mathbb{V})_{\gamma}$. Applying this to the arrows in $C_{\alpha}$ yields that $X_u = X_v$ (resp. $Y_u = Y_v)$ for all vertices $u,v$ of $C_{\alpha}$, and that each of these $\FF_1$-vector spaces is invariant under $\mathcal{F}(\mathbb{V})_{\alpha} = f_1$. Since the quiver formed from deleted $\alpha$ and $\beta$ from $S$ is still connected, applying this condition to every arrow in $S_1\setminus\{\alpha,\beta\}$ implies that $X_u = X_v$ and $Y_u = Y_v$ for all $u,v \in S_0$. Applying the condition at $\alpha$ implies that these $\FF_1$-linear subspaces are invariant under $\mathcal{F}(\mathbb{V})_{\alpha} = f_1$, and applying the condition at $\beta$ implies that they are invariant under $\mathcal{F}(\mathbb{V})_{\beta} = f_2$. In other words, the decomposition $\mathcal{F}(\mathbb{V}) = X\oplus Y$ induces a direct sum decomposition of $V$ which is $f_1$-invariant and $f_2$-invariant. Since $\mathbb{V}$ is indecomposable, we must have either $X = 0$ or $Y = 0$, from which it follows that $\mathcal{F}$ preserves indecomposables. 

Finally, we conclude that $\NI_{\wild_2}(n) \le \NI_{S}(|S_0|n)$ for all $n$, so that $\wild_2 \le_{\nil} S$. But $S$ is a subquiver of $Q$ and hence $S \le_{\nil} Q$, so that transitivity implies $\wild_2 \le_{\nil}Q$. Combining this with Theorem \ref{theorem: upper bound} yields $Q \approx_{\nil} \wild_2$.
\end{proof}  

\begin{rmk} 
We suspect that a similar statement holds in general for $Q$ and $\wild_n$, where  
\[
n \le \dim_{\ZZ_2}{H_1(\overline{Q},\ZZ_2)}. 
\]
We have not included it here because the $n = 2$ case is sufficient for the purposes of this document. 
\end{rmk}

\noindent This theorem immediately implies the following corollary:
 
\begin{cor} \label{corollary: quiver} 
Let $Q$ be a connected quiver with underlying graph $\overline{Q}$ such that $Q \le_{\nil} \wild_1$. Then $\dim_{\ZZ_2}H_1(\overline{Q},\ZZ_2) \le 1$.
\end{cor}

Let $Q$ be a quiver satisfying the hypotheses of Corollary \ref{corollary: quiver}. We note that $Q$ is a tree if and only if $H_1(\overline{Q},\ZZ_2) = 0$. If $\dim_{\ZZ_2}H_1(\overline{Q},\ZZ_2) = 1$ then $\overline{Q}$ must have a simple cycle which we call $C$. Since $\overline{Q}$ has only one cycle, the graph $\overline{Q}\setminus C$ formed by deleting $C$ and all edges adjacent to it must be a forest. In other words, $\overline{Q}$ can be formed from $C$ by gluing a tree to each vertex of $C$ (note that this tree may be a single vertex). If $v \in C_0$ is a vertex, we let $T_v$ denote the corresponding tree. The vertex $v$ satisfies $\{ v\} = (T_v)_0 \cap C_0$ and is called the \emph{root} of $T_v$. Note that if $u$ and $v$ are distinct vertices of $C$ then $T_u$ and $T_v$ are disjoint subgraphs of $\overline{Q}$. The graph $\overline{Q}$ is known as a \emph{pseudotree}.  \par \medskip 

\begin{mydef} 
Let $Q$ be a pseudotree with (unoriented) cycle $C$. We say that $Q$ is a \emph{proper pseudotree} if $T_v$ is not a point, for some $v \in C_0$.
\end{mydef}

The following lemma demonstrates a useful way to obtain direct sum decompositions for tree representations.

\begin{lem}\label{lemma: tree decomp}
Let $T$ be an oriented tree and with root $v$. Let $M$ be a nilpotent $\FF_1$-representation of $T$. If $M_v = A \oplus B$ is a direct sum decomposition of $\FF_1$-vector spaces, then there is a direct sum decomposition $M = X\oplus Y$ of $T$-representations such that $X_v = A$ and $Y_v = B$.
\end{lem}  
\begin{proof} 
We prove this by induction on $|T_0|$. If $|T_0| = 1$ then $T$ has no arrows and the claim is clear. Suppose the claim holds for all trees $T'$ with $|T'_0| < n$ and that $|T_0| = n$. Every tree has at least two leaves, so we can find a leaf $u \in T_0$ different from $v$. Let $T'$ be the tree formed from deleting $u$ and the arrow incident to it. Then $|T'_0|<n$ and the restricted representation $\Res_{T'}(M)$ decomposes as  
\[
\Res_{T'}(M)= X'\oplus Y', 
\]  
with $X'_v = A$ and $Y'_v = B$. If we define $X_w = X'_w$ and $Y_w = Y'_w$ for all $w \in T_0\setminus\{u\}$, then we will be done if we can find a direct sum decomposition 
\[ 
M_u = X_u \oplus Y_u 
\] 
such that $M = X\oplus Y$ as $T$-representations. Let $\alpha$ denote the arrow incident to $u$, and define $f$ to be the map $M_{s(\alpha)} \rightarrow M_{t(\alpha)}$ associated to $\alpha$.  There are two cases to consider: \\
\underline{Case 1:} Suppose that $s(\alpha) = u$. Let  
\[
X_u = f^{-1}(X_{t(\alpha)}\setminus\{0\}) \cup \ker (f), 
\] 
\[ 
Y_u = f^{-1}(Y_{t(\alpha)}\setminus\{0\}) \cup \{0\}.
\] 
Then $X_u$ and $Y_u$ are $\FF_1$-subspaces of $M_u$ such that $M_u = X_u \oplus Y_u$, $f(X_u) \subseteq X_{t(\alpha)}$ and $f(Y_u) \subseteq Y_{t(\alpha)}$. In other words, $X$ and $Y$ are sub-representations of $M$ and $M = X\oplus Y$. \\ 
\underline{Case 2:} Suppose that $t(\alpha) = u$. Let  
\[
X_u = f(X_{s(\alpha)}), 
\] 
\[ 
Y_u = (M_u \setminus f(X_{s(\alpha)})) \cup \{ 0\} .
\] 
Then $X_u$ and $Y_u$ are $\FF_1$-subspaces of $M_u$ such that $M_u = X_u \oplus Y_u$, $f(X_{s(\alpha)}) \subseteq X_u$ and $f(Y_{s(\alpha)}) \subseteq Y_u$. In other words, $X$ and $Y$ are sub-representations of $M$ with $M = X\oplus Y$.
\end{proof}

\begin{cor} 
Let $Q$ be a connected quiver that is obtained by gluing a tree $T$ to a subquiver $S$ along a root $v \in T_0$. Then the restriction functor $\Res_S$ preserves indecomposability.
\end{cor} 
\begin{proof} 
Let $M$ be an indecomposable $\FF_1$-representation of $Q$. If $\Res_S(M) = X'\oplus Y'$ as $S$-representations then in particular $M_v = X'_v \oplus Y'_v$ as $\FF_1$-vector spaces. By Lemma \ref{lemma: tree decomp} we can now find a decomposition $\Res_T(M) = X''\oplus Y''$ with $X''_v = X'_v$ and $Y''_v = Y'_v$. This induces a direct sum decomposition $M = X\oplus Y$ of $Q$-representations, where $X$ (resp. $Y$) satisfies $\Res_S(X) = X'$ and $\Res_T(X) = X''$ (resp. $\Res_S(Y) = Y'$ and $\Res_T(Y) = Y''$). By the indecomposability of $M$ we have $X = 0$ or $Y = 0$, which in turn implies $X' = 0$ or $Y' = 0$.
\end{proof}    

We are now in a position to study representations of pseudotrees more deeply. This requires a careful study of indecomposable modules over quivers of type $\tilde{\mathbb{A}}_n$. First, we generalize a construction from  \cite[Section 11]{szczesny2011representations} to the acyclic case. 

%Generalizing Szczesny's cycle representations: 
\begin{construction}\label{construction: strings} Let $C$ be a quiver of type $\tilde{\mathbb{A}}_{l}$. Order the vertices of $C$ cyclically as $1,\ldots , l$ with the subscripts understood to be taken mod $l$, as needed; let $\alpha_i$ denote the arrow connecting $i$ to $i+1$. With respect to this cyclic ordering, we can define a family of indecomposable representations $I_{[n,i]}$ as follows: For each $1 \le j \le l$, let  
\[
(I_{[n,i]})_j = \{ k \mid 1\le k \le n, k \equiv j - i+1 \text{ }(\operatorname{mod} l)\} \cup \{ 0 \}.  
\]
For each $k \in (I_{[n,i]})_j$ with $1 \le k < n$, note that $k+1 \in (I_{[n,i]})_{j+1}$. We define the map $f_j$ corresponding to $\alpha_j$ according to one of two cases: \\  
\noindent \underline{Case 1}: Suppose that $s(\alpha_j) = j$. Then for each $k \in (I_{[n,i]})_j$ with $1 \le k < n$, we have $f_j(k) = k+1$; if $n \in (I_{[n,i]})_j$ then $f_j(n) = 0$.\\ 
\noindent \underline{Case 2:} Suppose that $t(\alpha_j) = j$. Then $s(\alpha_j) = j+1$, and for each $k \in (I_{[n,i]})_{j+1}$ with $k > 1$ we define $f_j(k) = k-1$. If $1 \in (I_{[n,i]})_{j+1}$ then $f_j(1) = 0$. \\ 
Let $\Gamma = \Gamma_{I_{[n,i]}}$. Note that $\overline{\Gamma}$ is a path of length $n-1$, so that $I_{[n,i]}$ is indecomposable. We call the vertex of $\Gamma$ corresponding to $1$ the \emph{start} of $\Gamma$ and the vertex corresponding to $n$ the \emph{end} of $\Gamma$. If $C$ is equioriented, then the $I_{[n,i]}$ are precisely the representations described in  \cite{szczesny2011representations}.   
\end{construction}  

%Another generalization:  
\begin{construction}\label{construction: bands}
Let $C$ be an \emph{acyclic} quiver of type $\tilde{\mathbb{A}}_l$, with the vertices and arrows labeled as before. Then for each natural number $w$, we define another family of indecomposable representations $\tilde{I}_{w}$. These can be described through their quiver as follows: first start with the coefficient quiver of $I_{[wl,i]}$ described above. Then define a new representation $\tilde{I}_{[wl,i]}$ by attaching an $\alpha_{i-1}$-colored arrow between $1$ and $wl$, call it $\alpha$. We specify $s(\alpha) = 1$ if $s(\alpha_{i-1}) = i-1$ and $s(\alpha) = wl$ if $t(\alpha_{i-1}) = i-1$. In terms of the maps on arrows, this simply means that we require $f_{i-1}(1) = wl$ if $s(\alpha_{i-1}) = i-1$ and $f_{i-1}(wl) = 1$ if $t(\alpha_{i-1})=i-1$. Note that since $C$ is assumed to be acyclic, $\tilde{I}_{[wl,i]}$ is nilpotent. Furthermore, note that the coefficient quiver of $\tilde{I}_{[wl,i]}$ is obtained by winding around $C$ in a fixed direction $w$ times, so that $\tilde{I}_{[wl,i]} \cong \tilde{I}_{[wl,j]}$ for all $i, j \in C_0$. Hence we can define $\tilde{I}_w := \tilde{I}_{[wl,i]}$ unambiguously. 
 \end{construction}   
  
 %Classification of representations over cycle quivers: 
 \begin{mytheorem} 
Let $C$ be a quiver of type $\tilde{\mathbb{A}}_l$. If $C$ is acyclic, then any indecomposable nilpotent $\FF_1$-representation of $C$ is isomorphic to either $I_{[n,i]}$ or $\tilde{I}_n$ for some natural number $n$. If $C$ is equioriented, then the indecomposable nilpotent representations are all isomorphic to $I_{[n,i]}$ for some choice of $n$ and $i$. In particular, $C \approx_{\nil} \wild_1$.
 \end{mytheorem} 
 \begin{proof} 
Let $n$ be a natural number, with $M$ an indecomposable nilpotent $\FF_1$-representation of $C$ with dimension $n$. If $C$ is equioriented then the result follows from the classification of indecomposables given in \cite[Section 11]{szczesny2011representations}, so without loss of generality we may assume that $C$ is acyclic. In particular, $C$ has no loops. Let $\Gamma_M$ be the associated coefficient quiver. Note that since any vertex of $C$ is incident to two arrows, any vertex of $\overline{\Gamma}_M$ has valence at most two. It follows that $\overline{\Gamma}_M$ is either a line graph or a cycle. If $\overline{\Gamma}_M$ is a line graph then it is isomorphic to $I_{[n,i]}$ for some vertex $i$. Otherwise, it is a cycle and is isomorphic to $\tilde{I}_{n/l}$ (note that this can only happen when $n$ is divisible by $l$). It follows that for all $n$, 
\[ 
\NI_C(n) \le |C_0|+1 = (|C_0|+1) \cdot 1 = (|C_0|+1)\NI_{\wild_1}(n).
\] 
In particular, we have $C \le_{\nil}\wild_1$. Since $I_{[n,i]}$ exists for each $n\geq 1$, we also have
\[ 
\NI_C(n) \geq 1 = \NI_{\wild_1}(n).
\] 
Combining these inequalities yields $C \approx_{\nil} \wild_1$.
\end{proof} 

We now show that proper pseudotrees are not of bounded representation type. %To do this, we will need the following definition.

%Definition we'll need to show that proper pseudotrees grow faster than L_1: 
%\begin{mydef} 
%Let $T$ be a tree with $v \in T_0$. We say that an object $M \in \Rep(T,\FF_1)_{\nil}$ is \emph{$v$-full} if $t =\dim_{\FF_1}(M_v)$ is the number of indecomposable summands of $M$; i.e. if $M \cong M_1 \oplus \cdots \oplus M_t$, where each $M_i$ is an indecomposable.
%\end{mydef}  

%Pseudotrees grow faster than L_1: 
\begin{cor} 
Let $Q$ be a proper pseudotree with cycle $C$. Then there exists a natural number $K$ such that for all $n$, $\NI_Q(Kn) \geq n$. In particular, $Q\not\le_{\nil}\wild_1$.
\end{cor}  
\begin{proof}
%\textcolor{red}{I changed this WHOLE proof. I looked through the proof and the 3 isn't doing anything essential, so I deleted it. I had originally put it there to ensure that the coefficient quivers in the sequence were all non-isomorphic, but Condition (3) guarantees this (we only want to consider coefficient quivers up to coefficient isomorphism anyways).}
Put a cyclic ordering $1,\ldots , l$ on the vertices of $C$. Without loss of generality we may assume that $T_{1}$ is not a point. Define $T := T_{1}$, and let $S$ denote the simple nilpotent $T$-representation with $S_1 = [1]$, $S_i = 0$ otherwise and zero maps at each arrow. Furthermore, define $M$ to be the $2$-dimensional indecomposable representation of $T$ constructed as follows: since $T$ is not a point, there is necessarily an arrow $\alpha$ in $T$ adjacent to $1$. Then $M$ is defined by $M_{s(\alpha)} = M_{t(\alpha)} = [1]$, $f_{\alpha} = \id_{[1]}$, and $f_{\beta} = 0$ for all other arrows $\beta \in T_1$.

 To prove the first claim, we show that for each natural number $n$, we can construct a sequence  
 \[
 \tilde{M}^{(0)}, \tilde{M}^{(1)},\ldots , \tilde{M}^{(n-1)} 
 \]
  of nilpotent $\FF_1$-representations of $Q$ such that for each $i \le n-1$, 
\begin{enumerate} 
\item $\tilde{M}^{(i)}$ is indecomposable,
\item $\dim_{\FF_1}(\tilde{M}^{(i)}) = ln$, 
\item$\Res_T(\tilde{M}^{(i)}) \cong M^{\oplus i}\oplus S^{\oplus n-i}$. 
\end{enumerate} 
Conditions (1)-(3) together imply that $\NI_Q(l n) \geq n$, which proves the first claim. The claim $Q \not\le_{\nil} \wild_1$ will then follow from the fact that $\NI_{\wild_1}(n)  = 1$ for all $n \geq 1$.

% Fix a natural number $n$, and let $v_1$ denote the unique $v_1$-colored vertex of $\Gamma_M$.   Since $T$ is not a point, there is an arrow $\alpha$ in $T$ adjacent to $v_1$. We can then construct a $2$-dimensional indecomposable representation $M$ of $T$ such that $M_{s(\alpha)} = M_{t(\alpha)} = [1]$, $f_{\alpha} = \id_{[1]}$, and $f_{\beta} = 0$ for all other arrows $\beta \in T_1$. 

Define $\tilde{M}^{(0)}$ to be the representation $\tilde{M}^{(0)} = I_{[nl,1]}$, where the $C$-representation $I_{[nl,1]}$ is considered as a $Q$-representation via the functor of Lemma \ref{lemma: l.subq}. Note that $\Res_{T}(\tilde{M}^{(0)}) \cong S^{\oplus n }$, so $\tilde{M}^{(0)}$ satisfies Conditions (1)-(3) (with $i = 0$). Let $\tilde{\Gamma}^{(0)}$ denote the coefficient quiver associated of $\tilde{M}^{(0)}$, and let  
\[
\{ 1, l+1, \ldots , l(n-1) + 1 \} 
\]
 denote the first $n$ nonzero elements of $(\tilde{M}^{(0)})_{1}$ as we wind along the path from the start to the end.

There is a unique $1$-colored arrow in $\Gamma_M$ which we denote by $v$. Using the construction in Definition \ref{definition: gluing}, define $\tilde{\Gamma}^{(1)}$ to be the quiver obtained by deleting the vertex $nl$ from the amalgam $\tilde{\Gamma}^{(0)}\sqcup_{1\sim v}\Gamma_M$. Note that $\tilde{\Gamma}^{(1)}$ is connected, and the map $\tilde{\Gamma}^{(1)}\rightarrow Q$ described in Definition \ref{definition: gluing} is a winding: indeed, $\tilde{\Gamma}^{(0)}$ and $\Gamma_M$ have no arrows with the same color, so no subquivers of the form 

\[
\bullet \xrightarrow[]{\beta} \bullet \xleftarrow[]{\beta} \bullet\text{ or } \bullet \xleftarrow[]{\beta} \bullet \xrightarrow[]{\beta} \bullet
\] 

\noindent can be formed by the gluing. This means that $\tilde{\Gamma}^{(1)}$ defines an indecomposable $\FF_1$-representation of $Q$, which we denote by $\tilde{M}^{(1)}$. The construction of $\tilde{\Gamma}^{(1)}$ ensures that Conditions (1)-(3) (with $i = 1$) hold for $\tilde{M}^{(1)}$. For each remaining $1 < i \le n-1$ we construct $\tilde{\Gamma}^{(i)}$ by deleting the vertex $ln-i+1$ from the amalgam $\tilde{\Gamma}_{i-1}\sqcup_{1+l(i-1)\sim v}\Gamma_M$. Note that $1+l(i-1)$ is a vertex in $\tilde{\Gamma}^{(0)}$: since $\tilde{\Gamma}^{(0)}$ and $\Gamma_M$ share no arrows the the same color, the map $\tilde{\Gamma}^{(i)} \rightarrow Q$ described in Definition \ref{definition: gluing} remains a winding. Hence, $\tilde{\Gamma}^{(i)}$ defines an $\FF_1$-representation $\tilde{M}^{(i)}$ which satisfies Conditions (1)-(3) by construction. 
\end{proof}  

%%%%%%%%%%%%%%%%(TAME WILD STUFF)%%%%%%%%%%%%%

\subsection{Quivers of unbounded type over $\FF_1$}

It is known that a quiver $Q$ is tame if and only if $\overline{Q}$ is a simply-laced Dynkin diagram ($\mathbb{A}_n$, $\mathbb{D}_n$, $\mathbb{E}_6$, $\mathbb{E}_7$ or $\mathbb{E}_8$) or extended Dynkin diagram ($\tilde{\mathbb{A}}_n$, $\tilde{\mathbb{D}}_n$, $\tilde{\mathbb{E}}_6$, $\tilde{\mathbb{E}}_7$ or $\tilde{\mathbb{E}}_8$). As a simple consequence of our results, tame quivers are of bounded representation type over $\FF_1$. Quivers which are not tame are called wild; loosely speaking, their representation theory is as difficult as classifying pairs of square matrices up to simultaneous conjugation. Over $\FF_1$, every quiver $Q$ which is not of bounded type admits a fully faithful, indecomposable-preserving embedding of either $\wild_2$ or a pseudotree. If such a $Q$ is not a pseudotree, then Proposition \ref{proposition: rank} also implies that $Q \approx_{\nil} \wild_2$. These observations serve to partition quivers according to the complexity of their $\FF_1$-representations.

The theorem below summarizes our classification results for quivers of bounded representation type.  

\begin{mytheorem} \label{theorem: cycle or tree if and only if bdd rep}
Let $Q$ be a connected quiver. Then $Q$ is of bounded representation type if and only if $Q$ is either a tree or of type $\tilde{\mathbb{A}}_n$. Moreover, $Q$ is a tree quiver if and only if $Q \approx_{\nil} \wild_0$, and $Q$ is of type $\tilde{\mathbb{A}}_n$ if and only if $Q \approx_{\nil} \wild_1$.
\end{mytheorem} 

\noindent For every other quiver, we have the following result:

\begin{mytheorem}\label{theorem: not bdd type}
Suppose that $Q$ is a connected quiver that is not of bounded representation type. Then there exists a fully faithful, exact, indecomposable-preserving functor  
\[
\Rep(Q',\FF_1)_{\nil} \rightarrow \Rep(Q,\FF_1)_{\nil}, 
\] 
where $Q'$ is either a proper pseudotree or $\wild_2$. If $Q$ is not a pseudotree, then $Q \approx_{\nil} \wild_2$.
\end{mytheorem}   

\begin{rmk} 
Note that in this last theorem, you can have embeddings of \emph{both} a proper pseudotree and $\wild_2$. 
\end{rmk} 

We are very close to having described the equivalence classes of connected quivers with respect to the $\approx_{\nil}$ relation. The only equivalence classes which have not been fully described are those of the proper pseudotrees. In particular, it is not clear whether all proper pseudotrees lie in the same equivalence class: if they do, it is not clear whether this class includes $\wild_2$. Hence, we end this section with two questions aimed at future exploration. 

\begin{question} 
Let $Q$ and $Q'$ be proper pseudotrees. Does it follow that $Q \approx_{\nil} Q'$?
\end{question} 

\begin{question} 
Are there any pseudotrees $Q$, with $Q \approx_{\nil} \wild_2$?
\end{question}

%%%%%%%%%%%%%%%%%%%%(HALL ALGEBRAS OF L_n)%%%%%%%%%%%%
	
\section{Hall algebras of  full subcategories $\Rep(Q,\FF_1)_{\nil}$} \label{section: hall algebra of full subcategories}

%\begin{mythm}(\cite[Theorem 6 and Remark 2]{szczesny2011representations})\label{theorem: Matt's theorem on Hall algebra}
%Let $Q$ be a quiver and $\mathcal{C}=\Rep(Q,\FF_1)_{\nil}$. Then, the Hall algebra $H_\mathcal{C}$ is a graded, connected, and co-commutative Hopf algebra. Furthermore, $H_\mathcal{C}$ is isomorphic to the universal enveloping algebra $\mathbf{U}(\mathfrak{n}_Q)$ of the Lie algebra $\mathfrak{n}_Q$ spanned by $\delta_{[M]}$ for $M \in \textrm{Iso}(\mathcal{C})$ indecomposable.\footnote{In $H_\mathcal{C}$, primitives are precisely indecomposables.}
%\end{mythm}

In this section, we explore Hall algebras arising from full subcategories of $\Rep(Q,\FF_1)_{\nil}$ focusing on the case when $Q=\wild_n$. We then relate the Hall algebra of skew shapes as in \cite{szczesny2018hopf} to the coalgebra of a certain full subcategory of $\Rep(\wild_n,\FF_1)_{\nil}$. 

Let $Q$ be a connected quiver and $H_\textrm{$Q$,nil}$ be the Hall algebra of $\Rep(Q,\FF_1)_{\nil}$ - we know from \cite[Theorem 6]{szczesny2011representations} that $H_\textrm{$Q$,nil}$ exists and is given by
\[
H_\textrm{$Q$,nil} \simeq \mathbf{U}(\mathfrak{cq}_n),
\]
where the Lie algebra $\mathfrak{cq}_n$ is spanned by $\{\delta_{[\mathbb{V}]}\}$ for indecomposable nilpotent $\FF_1$-representations $\mathbb{V}$ and $\mathbf{U}(\mathfrak{cq}_n)$ is the universal enveloping algebra of $\mathfrak{cq}_n$. Moreover, it follows from Lemma \ref{l.basicprop} and Proposition \ref{proposition: colored quiver} that there is one-to-one correspondence between indecomposable nilpotent $\FF_1$-representations $\mathbb{V}$ of $Q$ and connected, acyclic windings $c:\Gamma \to Q$. Therefore, we have the following theorem. 

\begin{mythm}\label{theorem: hall algebra for $L_n$}
With the same notation as above, the Hall algebra $H_\textrm{$Q$,nil}$ is isomorphic to the enveloping algebra $\mathbf{U}(\mathfrak{cq}_n)$. The Lie algebra $\mathfrak{cq}_n$ may be identified with 
\[
\mathfrak{cq}_n=\{\delta_{\mathbb{V}_{\Gamma}} \mid \textrm{ connected, acylic windings $c:\Gamma \to Q$ }\}
\]
with Lie bracket given by
\begin{equation}\label{eq: Lie bracket for $H_C$}
[\delta_{\mathbb{V}_{\Gamma}},\delta_{\mathbb{V}_{\Gamma'}}]=\delta_{\mathbb{V}_{\Gamma}}*\delta_{\mathbb{V}_{\Gamma'}}-\delta_{\mathbb{V}_{\Gamma'}}*\delta_{\mathbb{V}_{\Gamma}}.
\end{equation}
\end{mythm}

Note that from the algebra structure \eqref{eq: hall product} of $H_\textrm{$Q$,nil}$, we have
\begin{equation}
\delta_{\mathbb{V}_{\Gamma}}*\delta_{\mathbb{V}_{\Gamma'}} = \sum_{\Gamma''} \mathbf{a}_{\Gamma,\Gamma'}^{\Gamma''}\delta_{\mathbb{V}_{\Gamma''}}
\end{equation}
For connected, acylic windings $c:\Gamma \to Q$, the product $\delta_{\mathbb{V}_{\Gamma}}*\delta_{\mathbb{V}_{\Gamma'}}$can be intuitively considered to involve all ways of stacking the coefficient quiver $\Gamma'$ on top of the coefficient quiver $\Gamma$ to obtain another coefficient quiver $\Gamma''$\footnote{See Lemma \ref{lem: quotient rep}~(3).} as well as disjoint union. Construction \ref{construction: GammaFM} shows one of such ``stacking'' operation when $Q=\wild_n$ (see also a short exact sequence \eqref{eq: exact seq gluing}). \par \medskip

Now, we consider the case when $Q=\wild_n$. We first prove that the symmetric group $S_n$ on $n$ letters induces a natural automorphism of $H_\textrm{$\wild_n$,nil}$. To be precise, let $\mathbb{V}=(V,f_1,\dots,f_n) \in \Rep(\wild_n,\FF_1)_{\nil}$. For each $\sigma \in S_n$, we define the following representation
\begin{equation}\label{eq: $S_n$-action}
\mathbb{V}^\sigma=(V,g_1,\dots,g_n),\textrm{ where } g_i:=f_{\sigma(i)}.
\end{equation}
In terms of the associated coefficient quiver, $\Gamma_{\mathbb{V}^\sigma}$ is obtained from $\Gamma_{\mathbb{V}}$ by switching $\alpha$-colored arrows with $\sigma(\alpha)$-colored arrows. The coefficient quiver $\Gamma_{\mathbb{V}^\sigma}$ is acylic, and hence $\mathbb{V}^\sigma$ is indeed a nilpotent representation of $\wild_n$ over $\FF_1$. Furthermore, it is clear that $\mathbb{V}$ is indecomposable if and only if $\mathbb{V}^\sigma$ is indecomposable since indecomposability is defined in terms of the underlying graph of a coefficient quiver. 

\begin{pro}
	With the same notation as above, each $\sigma \in S_n$ as in \eqref{eq: $S_n$-action} induces a Hopf algebra automorphism of the Hall algebra $H_\textrm{$\wild_n$,nil}$ of $\Rep(\wild_n,\FF_1)_{\nil}$. In particular, we have a natural group homomorphism $S_n \to \textrm{Aut}(H_\textrm{$\wild_n$,nil})$. 	
\end{pro}
\begin{proof}
	For $\sigma \in S_n$, let $\phi_\sigma: H_\textrm{$\wild_n$,nil} \to H_\textrm{$\wild_n$,nil}$ sending an indecomposable $\mathbb{V}$ to $\mathbb{V}^\sigma$ and extend to all of $H_\textrm{Q,nil}$. It is clear that $\phi_\sigma$ is an isomorphism of coalgebras from the Krull-Schmidt property proven in \cite[Theorem 4]{szczesny2012hall}. Therefore, we only have to prove that $\phi_\sigma$ is a map of algebras. But, this is also straightforward since for any indecomposables $M,N \in \Rep(\wild_n,\FF_1)_{\nil}$ and a representation $R$ of $\wild_n$ over $\FF_1$, 
	\[
	0 \to N \to R \to M \to 0
	\] 
	is exact if and only if the following is exact
	\[
	0 \to N^\sigma \to R^\sigma \to M^\sigma \to 0.
	\] 
\end{proof}

Now, we introduce the notion of a path monoid. By a monoid, we will always mean a multiplicative monoid $M$ (not necessarily commutative) with an absorbing element $0$, i.e., $a\cdot 0 =0 \cdot a =0$ for all $a \in M$. 

\begin{mydef}(Path monoid)\label{definition: quiver monoid}
Let $Q$ be a quiver. The \emph{path monoid} $M_Q$ associated to $Q$ is defined by $M_Q:=F/\sim$, where $F$ is the free monoid generated by $\{e_i\}_{i \in Q_0}\cup \{\alpha\}_{\alpha \in Q_1}$, and $\sim$ is a congruence relation generated by the following:
\begin{equation}
e_i^2\sim e_i, \quad e_ie_j\sim 0~ (i\neq j), \quad e_{t(\alpha)}\alpha\sim\alpha e_{s(\alpha)}\sim\alpha,
\end{equation}
and
\[
\alpha\beta \sim 0 \textrm{ whenever $t(\alpha) \neq s(\beta)$.}
\]
\end{mydef}

In the classical setting, the category of representations of a finite quiver $Q$ over a field $k$ is equivalent to the category of left $kQ$-modules, where $kQ$ is the path algebra associated to $Q$. In the case of representations over $\FF_1$, it seems to be difficult to obtain a similar results since there is no ``partition of unity by projections'' for monoids due to the lack of additive structure. Nonetheless, we have the following. 

\begin{pro}\label{proposition: fully faithful functor}
Let $Q$ be a quiver, then there exists a fully faithful functor from the category of representations of $Q$ over $\FF_1$ to the category of finite left $M_Q$-modules. 
\end{pro}
\begin{proof}
Let $\mathbb{V}=(V_i,f_\alpha)$ be a representation of $Q$. Consider the following $\FF_1$-vector space:
\[
V:=\bigoplus_{i \in Q_0}V_i
\]
We first claim that $V$ is a left $M_Q$-module. In fact, for each $i \in Q_0$, let $g_i:V \to V$ be the composition $V \twoheadrightarrow V_i \hookrightarrow V$. Similarly, for each $\alpha \in Q_1$, we let $g_\alpha:V \to V$ be the following composition:
\[
V \twoheadrightarrow V_{s(\alpha)}\overset{f_\alpha}\longrightarrow V_{t(\alpha)} \hookrightarrow V.
\]
One can easily check that
\begin{equation}\label{eq: relations}
g_i^2= g_i, \quad g_ig_j= 0~ (i\neq j), \quad g_{t(\alpha)}g_\alpha=g_\alpha g_{s(\alpha)}=g_\alpha.
\end{equation}
Let $M_Q=F/\sim$ as in Definition \ref{definition: quiver monoid}. It follows from \eqref{eq: relations} that we have a monoid morphism $M_Q \to \textrm{End}(V)$ sending $\overline{e_i}$ to $g_i$ and $\overline{\alpha}$ to $g_\alpha$, where $\overline{e_i}$ and $\overline{\alpha}$ are the equivalence classes in $M_Q$. In particular, $V$ is a left $M_Q$-module. We let $\mathbf{F}(\mathbb{V})$ denote this left $M_Q$-module.

Let $\Phi:\mathbb{V}=(V_i,f_\alpha) \to \mathbb{W}=(W_i,g_\alpha)$ be a morphism of representations of $Q$ over $\FF_1$. In other words, we have a family of maps $\phi_i:V_i \to W_i$, for $i \in Q_0$, which satisfy a certain compatibility condition. Consider the following:
\[
\mathbf{F}(\Phi):\bigoplus_{i \in Q_0}V_i \to \bigoplus_{i \in Q_0}W_i, \quad v_i \mapsto \mathbf{F}(\Phi)(v_i)=\phi_i(v_i), %elements of a direct sum of FF_1-spaces aren't n-tuples, so I figured the parentheses were confusing in this case
\] 
where $v_i \in V_i$. Then $\mathbf{F}(\Phi)$ is a morphism of left $M_Q$-modules, that is, we have
\begin{equation}
\mathbf{F}(\Phi)(\overline{e_i}\cdot v)=\overline{e_i}\cdot \mathbf{F}(\Phi)(v), \quad \mathbf{F}(\Phi)(\overline{\alpha}\cdot v)=\overline{\alpha}\cdot \mathbf{F}(\Phi)(v),
\end{equation}
in particular, $\mathbf{F}$ defines a functor from the category of representations of $Q$ over $\FF_1$ to the category of left $M_Q$-modules.

Next, we claim that $\mathbf{F}$ is fully faithful. In fact, from the definition, it is clear that $\mathbf{F}$ is faithful. For fullness, let $f: \mathbf{F}(\mathbb{V}) \to \mathbf{F}(\mathbb{W})$ be a morphism of left $M_Q$-modules. For $v=(v_i)_{i \in Q_0} \in \mathbf{F}(\mathbb{V})$, we have $f(\overline{e_i}\cdot v)=\overline{e_i}\cdot f(v)$. It follows that $f(V_i) \subseteq W_i$. Let $\phi_i:=f|_{V_i}$. Then, $\Phi=(\Phi_i) \in \Hom(\mathbb{V},\mathbb{W})$. In fact, we only have to check the compatibility condition, i.e., the commutativity of the following diagram: for each $i \in Q_0$: 
\begin{equation}\label{eq: commutative diagram maps}
\begin{tikzcd}[row sep=large, column sep=1.5cm]
V_{s(\alpha)}\arrow{r}{\phi_{s(\alpha)}}\arrow{d}{f_\alpha}
& W_{s(\alpha)} \arrow{d}{g_\alpha} \\
V_{t(\alpha)} \arrow{r}{\phi_{t(\alpha)}} 
& W_{t(\alpha)}
\end{tikzcd}
\end{equation}
But, this is clear since $f$ is a morphism of left $M_Q$-modules. 
\end{proof}

%In view of Corollary \ref{corollary: corollary for module correspondence}, the following is what one may expect as skew shapes correspond to certain $\mathbb{F}_1\langle x_1,\dots,x_n\rangle$-modules (called the type-$\alpha$ in \cite{szczesny2018hopf}) and this is a special case of representations of $\wild_n$ over $\mathbb{F}_1$.

The following definition is introduced by Szczesny.  

\begin{mydef}\cite[Definition 2.2.1]{szczesny2018hopf}
	Let $A$ be a commutative monoid, $M$ be an $A$-module. $M$ is said to be \emph{type-$\alpha$} if the following holds: for any $x\in A$, and $a,b \in M$,
	\[
	x\cdot a = x \cdot b \iff a=b \textrm{ or } x\cdot a = x \cdot b=0.
	\]
\end{mydef}

The following is a key to link coefficient quivers and certain combinatorial objects in \cite{szczesny2018hopf}.

\begin{cor}\label{corollary: corollary for module correspondence}
Let $Q=\wild_n$.
\begin{enumerate}
\item 
$M_Q$ is the free monoid generated by $\{x_1,\dots,x_n\}$. 
\item 
The category $\emph{Rep}(Q,\FF_1)$ is equivalent to the category of finite left $M_Q$-modules of type-$\alpha$. In particular, the full subcategory $\mathcal{C}$ consisting of representations $\mathbb{V}=(V,f_1,\dots,f_n)$ such that $f_if_j=f_jf_i$ is equivalent to the category of $\FF_1\langle x_1,\dots,x_n\rangle$-modules, where $\FF_1\langle x_1,\dots,x_n\rangle$ is a free commutative monoid generated by $\{x_1,\dots,x_n\}$.
\end{enumerate}	
\end{cor}
\begin{proof}
The first assertion is clear. For the second assertion, let $A:=\FF_1\langle x_1,\dots,x_n\rangle$ and $\mathbb{V}=(V,f_1,\dots,f_n)$ be an $\FF_1$-representation of $Q$. Since each $f_i:V \to V$ is an $\FF_1$-linear map, the following holds: for any $a,b \in V-\{0\}$, $i \in \{1,\dots,n\}$, a positive integer $k$,
\begin{equation}\label{eq: type alpha}
f_i^k(a)=f_i^k(b)\iff a=b \textrm{ or } f_i^k(a)=f_i^k(b)=0.
\end{equation}
From this, one can easily see that $\mathbf{F}(\mathbb{V})$ is a type-$\alpha$ $A$-module. Hence, it is enough to prove that the functor $\mathbf{F}$ in Proposition \ref{proposition: fully faithful functor} is essentially surjective onto the full subcategory whose objects are type-$\alpha$ modules. Let $V$ be a left $M_Q$-module which is type-$\alpha$. For each $\alpha \in Q_1$, we have an $\FF_1$-linear map $f_\alpha: V \to V$ since $\overline{\alpha}(V) \subseteq V$. Therefore, we have a representation $\mathbb{V}=(V,f_\alpha)$ of $\wild_n$. One can easily check that $\mathbf{F}(\mathbb{V})=V$.  
\end{proof}

\begin{rmk}
Let $Q$ be a quiver and $k$ be a field. Then,
\[
kQ \simeq k\otimes_{\mathbb{F}_1}M_Q,\quad  
\]
where $kQ$ is the path algebra. On the other hand, with the forgetful functor $\mathcal{U}$ in Example \ref{example: adjunction for scalar extensions}, the path monoid $M_Q$ is a sub-monoid of $\mathcal{U}(kQ)$. 
\end{rmk}

Next, we illustrate how Theorem \ref{theorem: hall algebra for $L_n$} is analogous to Theorem 6.0.1 of \cite{szczesny2018hopf}, under which our ``stacking operation'' of coefficient quivers becomes a ``stacking operation'' of combinatorial objects, called skew shapes. In this regard, one may view coefficient quivers as a certain non-commutative generalization of skew shapes.

Let's first recall the definition of skew shapes. We first define a canonical partial order on $\mathbb{Z}^n$ as follows:
\begin{equation}\label{eq: partion order on Zn}
(a_1,\dots,a_n) \leq (b_1,\dots,b_n) \iff a_i\leq b_i \quad \forall i=1,\dots,n.
\end{equation}
A skew shape is a sub-poset\footnote{By a sub-poset, we mean a subset with the induced partial order.} of $\mathbb{Z}^n$ which has been introduced by Szczesny to study modules over $\mathbb{F}_1\langle x_1,\dots,x_n\rangle$ satisfying certain conditions.  

\begin{mydef}\cite[Definition 5.1.1]{szczesny2018hopf}
	Let $S \subseteq \mathbb{Z}^n$ be a sub-poset. 
	\begin{enumerate}
		\item 
		$S$ is said to be an \emph{n-dimensional skew shape} if $S$ is finite and convex.\footnote{Note that ``convex'' is in the context of posets; it means that for $a,c \in S$ if $a\leq b\leq c$, then $b \in S$. We will omit ``n-dimensional'' whenever there is no possible confusion.}	
		\item 
		A skew shape $S$ is said to be \emph{connected} if $S$ is connected as a poset.\footnote{By this we mean that for any $a,b \in S$, there exists $a=c_1,\dots,c_r=b \in S$ such that $c_i$ and $c_{i+1}$ are comparable for all $i=1,\dots,r-1$.}
		\item 
		Skew shapes $S$ and $S'$ are said to be \emph{equivalent} if there is $\mathbf{x} \in \mathbb{Z}^n$ such that $S=S'+\mathbf{x}$. In other words, if $S$ is a translation of $S'$.\footnote{Clearly this defines an equivalence relation.} 	
	\end{enumerate}
\end{mydef}

\begin{myeg}
Let $S=\{(1,0),(2,0),(3,0),(4,0),(0,1),(1,1),(0,2)\} \subseteq \mathbb{Z}^2$. The following is an illustration of $S$ (up to translation by $k \in \mathbb{Z}^2$),
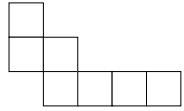
\[
\gyoung(;,;;,:;;;;) 
\]
\vspace{0.2cm}

Here are illustrations of other skew shapes.
\[
\gyoung(;,;;;,:;;;,::;;;) \qquad \gyoung(;;,:;)
\]
\vspace{0.2cm}
\end{myeg}

We recall another definition from \cite[Section 2.2.2]{szczesny2018hopf}. By a graded $\mathbb{Z}^n$-module, we mean a $\mathbb{Z}^n$-module $M$ with a decomposition 
\begin{equation}\label{eq: grading}
M=\bigoplus_{k \in \mathbb{Z}^n} M_k
\end{equation}
such that $0 \in M_0$ and $k\cdot M_m \subseteq M_{k+m}$, for $k \in \mathbb{Z}^n$. We say that a $\mathbb{Z}^n$-module $M$ admits a $\mathbb{Z}^n$-grading if it has structure \eqref{eq: grading}.

Szczesny showed that to a skew shape $S$, one can associate an $\FF_1\langle x_1,\dots,x_n\rangle$-module $M_S$. As a set $M_S:=S\sqcup \{0\}$, and $\FF_1\langle x_1,\dots,x_n\rangle$-action is given as follows: for $s \in M_S$ and $\mathbf{k} \in \mathbb{Z}_{\geq 0}^n$, 
\begin{equation}
x^\mathbf{k}\cdot s=\begin{cases}
	\mathbf{k}+s \textrm{ if $\mathbf{k}+s \in S$ }\\
	0 \textrm{ otherwise.}
\end{cases}
\end{equation}
Szczesny proved that the associated module $M_S$ is finite, $\mathbb{Z}^n$-graded, indecomposable\footnote{This is precisely when $S$ is connected.}, type-$\alpha$ $\FF_1\langle x_1,\dots,x_n\rangle$-module. In fact, Szczesny proved that the converse also holds as follows:

\begin{mythm}\cite[Theorem 5.4.2]{szczesny2018hopf}
Let $M$ be a finite, $\mathbb{Z}^n$-graded, indecomposable, type-$\alpha$ $\FF_1\langle x_1,\dots,x_n\rangle$-module such that $\emph{Supp}(M)=0$. Then, $M \simeq M_S$ for some connected skew shape $S$ in $\mathbb{Z}^n$. 
\end{mythm}

In what follows, we rephrase Szczesny' result in terms of $\FF_1$-representations of $\wild_n$. We start with the following definition.

\begin{mydef}
Let $\mathbb{V}=([m],f_1,f_2,\dots,f_n)$ be a nilpotent $\FF_1$-representation of $\wild_n$. We say that $\mathbb{V}$ \emph{admits a $\mathbb{Z}^n$-grading} if the following conditions hold:
\begin{enumerate}
	\item 
The corresponding module $\mathbf{}F(\mathbb{V})$ in Corollary \ref{corollary: corollary for module correspondence} admits a $\mathbb{Z}^n$-grading as above.
\item 
If $\{ e_i \mid i\le n\}$ denotes the standard basis of $\ZZ^n$, then there exists a $\sigma \in S_n$ such that $f_i$ has degree $e_{\sigma(i)}$ for all $i$. 
\end{enumerate}
\end{mydef}

\begin{pro}\label{proposition: correspondecne skew shape colored quivers}
There exists a one-to-one correspondence between  equivalence classes of connected skew shapes in $\mathbb{Z}^n$ and isomorphism classes of indecomposable, nilpotent $\FF_1$-representations $\mathbb{V}=([m],f_1,f_2,\dots,f_n)$ of $\wild_n$ such that $f_if_j=f_jf_i$ for $i,j=1,\dots,n$ and admitting a $\mathbb{Z}^n$-grading.
\end{pro}
\begin{proof}
Let $S\subseteq \mathbb{Z}^n$ be a skew shape. We claim that $S$ determines a nilpotent $\FF_1$-representation $\mathbb{V}=([m],f_1,f_2,\dots,f_n)$ of $\wild_n$ such that $f_if_j=f_jf_i$ for $i,j=1,\dots,n$ and admits a $\mathbb{Z}^n$-grading. Indeed, the same idea as in \cite{szczesny2018hopf} works here as follows: let $[m]=S \cup \{0\}$, and $\{e_1,\dots,e_n\}$ be the standard basis vectors in $\mathbb{Z}^n$. Define for each $a \in [m]-\{0\}$, 
\[
f_i^n(a)=\begin{cases}
	e_i+a \textrm{ if $e_i +a \in S$ }\\
	0 \textrm{ otherwise.}
\end{cases}
\]
Then, one can easily see that $\mathbb{V}=([m],f_1,\dots,f_n)$ is a desired nilpotent $\FF_1$-representation and any two equivalent skew shapes determines the same isomorphism class of $\FF_1$-representations. Furthermore, if a skew shape is connected then $\mathbb{V}$ is indecomposable. 

Conversely, let $\mathbb{V}=([m],f_1,f_2,\dots,f_n)$ be an indecomposable, nilpotent $\FF_1$-representation of $\wild_n$ such that $f_if_j=f_jf_i$ $\forall~i,j \in \{1,\dots,n\}$ admitting a $\mathbb{Z}^n$-grading. Let $A=\mathbb{F}_1\langle x_1,\dots,x_n \rangle$. From the proof of Corollary \ref{corollary: corollary for module correspondence}, we know that $[m]$ is a finite $A$-module which is type-$\alpha$ and admits a $\mathbb{Z}^n$-grading. To be precise, to each $a\neq 0 \in [m]$, $x_i$ acts as:
\[
x_i\cdot a :=f_i(a). 
\]
The assumption that $f_if_j=f_jf_i$ ensures that $A$ acts on $[m]$, and \eqref{eq: type alpha} is equivalent to the type-$\alpha$ condition. Let $M$ be this $A$-module structure on $[m]$. Since $\mathbb{V}$ is nilpotent, we have that
\[
\textrm{Ann}_A(M)=\langle x_1^{i_i},\dots,x_n^{i_n}\rangle, \quad \textrm{for some }i_1,\dots,i_n \in \mathbb{Z}_{>0}.
\]
In particular, $\textrm{Supp}(M)=0$. Finally, one can easily observe that $M$ is an indecomposable $A$-module since $\mathbb{V}$ is indecomposable. Now, it follows from \cite[Theorem 5.4.2]{szczesny2018hopf} that $M\simeq M_S$ for some connected skew shape $S \subseteq \mathbb{Z}^n$. If $\mathbb{V}\simeq \mathbb{W}$, then the corresponding skew shapes are equivalent since in this case they define an isomorphic type-$\alpha$ module. 

Finally, one may check that the above constructions are inverses to each other. 
\end{proof}

\begin{rmk}
Let $S$ be a skew shape in $\mathbb{Z}^n$. Then, one can construct a ``coefficient quiver'' $\Gamma_S$ associated to $S$ as follows: let $S$ be the set of vertices of $\Gamma_S$ which are each colored with the unique vertex of $\wild_n$. We draw an $\alpha_i$-colored edge from $v_1$ to $v_2$ if $v_1+e_i=v_2$ viewed as elements in $S$. Now, let $\mathbb{V}=([m],f_1,f_2,\dots,f_n)$ be an indecomposable, nilpotent $\FF_1$-representation of $\wild_n$ such that $f_if_j=f_jf_i$ for $i,j=1,\dots,n$, and $S_\mathbb{V}$ be the associated skew shape. One can easily observe that $\Gamma_{S_\mathbb{V}}=\Gamma_\mathbb{V}$. 
In particular, one may prove one direction of Proposition \ref{proposition: correspondecne skew shape colored quivers} as follows: given a skew shape $S$, we construct a coefficient quiver $c:\Gamma_S \to \wild_n$. Then, from Proposition \ref{proposition: colored quiver}, $\Gamma_S$ uniquely determines a representation $\mathbb{V}$, which is isomorphic to the one we obtained in Proposition \ref{proposition: correspondecne skew shape colored quivers}
\end{rmk}
%\begin{lem}\label{lemma: shape to rep}
%	A skew shape $S \subseteq \mathbb{Z}^n$ determines a representation $\mathbb{V}=([m],f_1,f_2,\dots,f_n) \in \emph{Rep}(\wild_n,\mathbb{F}_1)_{\textrm{nil}}$ such that $f_if_j=f_jf_i$ for $i,j=1,\dots,n$. Furthermore $\mathbb{V}$ admits a $\mathbb{Z}^n$-grading.
%\end{lem}
%\begin{proof}
%	The same idea as in \cite{szczesny2018hopf} works here as follows: Let $[n]=S \cup \{0\}$, and $\{e_1,\dots,e_n\}$ be the standard basis vectors in $\mathbb{Z}^n$. Define for each $a \in [n]-\{0\}$, 
%	\[
%	f_i^n(a)=\begin{cases}
%	e_i+a \textrm{ if $e_i +a \in S$ }\\
%	0 \textrm{ otherwise.}
%	\end{cases}
%	\]
%	Then, one can easily see that $\mathbb{V}=([n],f_1,\dots,f_n)$ is a desired nilpotent representation. %satisfying the conditions in Lemma \ref{lemma: rep to shape}. 
%\end{proof}

%By combining the above one has the following:

%\begin{pro}\label{proposition: correspondecne skew shape colored quivers}
%	There exists a one-to-one correspondence between  equivalence classes of connected skew shapes in $\mathbb{Z}^n$ and indecomposable representations $\mathbb{V}=([m],f_1,f_2,\dots,f_n) \in \emph{Rep}(\wild_n,\mathbb{F}_1)_{\textrm{nil}}$ such that $f_if_j=f_jf_i$ for $i,j=1,\dots,n$ and admitting a $\mathbb{Z}^n$-grading.
%\end{pro}

When it comes to Hall algebras, Proposition \ref{proposition: correspondecne skew shape colored quivers} implies the following.   

%\begin{cor}\label{corollary: skew shapes and colored quivers hall algebras}
%Let $\mathcal{C}_n$ be the full subcategory of $\textrm{Rep}(\wild_n,\FF_1)_{\nil}$ consisting of objects $\mathbb{V}=(V,f_i)$ such that $f_if_j=f_jf_i$ for $i,j=1,\dots,n$. Then, the Hall algebra $H_{\mathcal{C}_n}$ of $\mathcal{C}_n$ is well-defined and is isomorphic to the Hall algebra $SK_n$ of the category $\mathcal{D}=A$-$\textrm{mod}_0^{\alpha,gr}$ in \cite{szczesny2018hopf}
%\end{cor}

\begin{cor}\label{corollary: skew shapes and colored quivers hall algebras}
Let $\mathcal{C}_n$ be the full subcategory of $\emph{Rep}(\wild_n,\FF_1)_{\nil}$ consisting of objects $\mathbb{V}=(V,f_i)$ such that $f_if_j=f_jf_i$ for $i,j=1,\dots,n$ and admitting a $\mathbb{Z}^n$-grading. Let $X_n$ denote the linear subspace of $H_{\wild_n,\nil}$ generated by the objects of $\mathcal{C}_n$. Then $X_n$ is a subcoalgebra of $H_{\wild_n,\nil}$ isomorphic to the underlying coalgebra of $SK_n$, the Hall algebra of the category $\mathcal{D}=A$-$\textrm{mod}_0^{\alpha,gr}$ in \cite{szczesny2018hopf}
\end{cor}
\begin{proof}
This directly follows from Corollary \ref{corollary: corollary for module correspondence} and Proposition \ref{proposition: correspondecne skew shape colored quivers}. 
\end{proof}

\begin{rmk}\label{remark: quiver remark}
One may directly define the Hall algebra $H_n$ of $\mathcal{C}_n$ by using the same recipe as the Hall algebra $H_Q$ of $\Rep(Q,\FF_1)_{\nil}$. In this case, Corollary \ref{corollary: corollary for module correspondence} yields an isomorphism $H_n\simeq SK_n$ (as Hopf algebras). However, $H_n$ is not a Hopf subalgebra of $H_{\wild_n,\nil}$ for $n>1$. This is because, for $n>1$, $\mathcal{C}_n$ is not closed under extensions. This means that the coalgebra isomorphism in Corollary \ref{corollary: skew shapes and colored quivers hall algebras}
\[ 
X_n \cong SK_n 
\]
is not an algebra isomorphism. For instance, consider the short exact sequence in $\Rep(\wild_n,\FF_1)_{\nil}$ which appears below. Note that representations are written in terms of their coefficient quivers, with $f_1$ acting via blue (dotted) arrows, $f_2$ acting via red arrows, and the other maps acting as zero:
 \begin{equation}\label{eq: colored quiver}
0 \rightarrow
\begin{tikzcd}[arrow style=tikz,>=stealth,row sep=2em]
3 \arrow[d,blue,dotted,very thick]  \\  
4 
\end{tikzcd} \rightarrow 
\begin{tikzcd}[arrow style=tikz,>=stealth,row sep=2em]
1 \arrow[d,blue, dotted, very thick] \\  
2 \arrow[d,red,very thick] \\ 
3 \arrow[d,blue,dotted,very thick] \\ 
4
\end{tikzcd} \rightarrow 
\begin{tikzcd}[arrow style=tikz,>=stealth,row sep=2em]
1 \arrow[d,blue,dotted,very thick]  \\  
2 
\end{tikzcd} \rightarrow  
0.
\end{equation}
The two-dimensional representations on the endpoints are in $\mathcal{C}_n$, but the middle term is not. 
\end{rmk}

The case $n=1$ in Corollary \ref{corollary: skew shapes and colored quivers hall algebras} is better behaved. In this case, any representation $\mathbb{V}=(V,f) \in  \textrm{Rep}(\wild_1,\FF_1)_{\nil}$ satisfies the commutativity condition. %Furthermore, $(2)$ of Lemma \ref{l.basicprop} ensure that $\mathbb{V}$ satisfies the condition \eqref{eq: type alpha}. 
In particular, we have an isomorphism of Hopf algebras:
\[
H_{\mathcal{C}_1} \simeq H_\textrm{$\wild_1$,nil}.
\]

Recall that the ring $\Lambda$ of symmetric functions has a natural Hopf algebra structure; for instance, see \cite{borger2005plethystic}. The Hall algebra $SK_n$ is a generalization of $\Lambda$ in the sense that $SK_1 \simeq \Lambda$. Now, with Corollary \ref{corollary: skew shapes and colored quivers hall algebras}, we reprove the following: 

\begin{cor}(\cite[Theorem 9]{szczesny2011representations})
With the same notation as above, $H_\textrm{$\wild_1$,nil}$ is isomorphic to the ring $\Lambda$ of symmetric functions (viewed as a Hopf algebra). 
\end{cor}
\begin{proof}
	We may see $\Lambda$ as the Hall algebra $SK_1$. In this case, the category $\mathcal{C}:=A$-$\textrm{mod}_0^{\alpha,gr}$, where $A=\mathbb{F}_1\langle x \rangle$ is a full subcategory of $\textrm{Rep}(\wild_1,\FF_1)_{\nil}$ which is closed under taking subobjects, quotient objects, and extensions. In particular, one has a natural map $f:\Lambda \to H_{\wild_1}$ of Hopf algebras which can be easily seen to be injective from Proposition \ref{proposition: correspondecne skew shape colored quivers}.
\end{proof}

In \cite{szczesny2014hall}, Szczesny proves that for the category $\mathcal{C}^N_{nil}$ of nilpotent $\FF_1\langle t\rangle$-modules, the Hall algebra $H_{\mathcal{C}^N_{nil}}$ is isomorphic to the dual of the Kreimer's Hopf algebra of rooted forests \cite{kreimer1998hopf}. One can easily see that the categorical equivalence in Corollary \ref{corollary: corollary for module correspondence} restricts to the equivalence  $\textrm{Rep}(\wild_n,\FF_1)_{\textrm{nil}}$ and the category of finite nilpotent left $M_{\wild_n}$-modules. We thus have the following. 

\begin{cor}
With the same notation as above, $H_\textrm{$\wild_1$,nil}$ is isomorphic to the dual $H^*_K$ as Hopf algebras, where $H_K$ is Kreimer's Hopf algebra of rooted forests. 
\end{cor}

%one may notice that  the subcategory $\mathcal{C}_1$ is actually equivalent to $\textrm{Rep}(\wild_1,\FF_1)_{\nil}$; this is because any representation of $\wild_1$ over $\FF_1$ satisfies the conditions in Lemma \ref{lemma: rep to shape}. In particular, $H_\mathcal{C}$ is isomorphic to $H_\textrm{$\wild_n$,nil}$. 

%one may notice the equivalences of categories:
%\[
%A\textrm{-}\textrm{mod}_0^{\alpha,gr} \simeq \mathcal{C} \simeq \textrm{Rep}(\wild_1,\FF_1)_{\nil},
%\]
%where $A=\mathbb{F}_1\langle x \rangle$.and \cite{szczesny2018hopf}

%\begin{myeg}
%	In \cite{szczesny2011representations}, Szczesny proved that the Hall algebra of $\Rep(\wild_1,\FF_1)_{\nil}$ is isomorphic to the ring $\Lambda$ of symmetric functions as a Hopf algebra. We note that in fact, the Hall algebra of $\Rep(\wild_1,\FF_q)_{\nil}$ is isomorphic to $\Lambda$ (also as a Hopf algebra). See, \cite{schiffmann2006lectures}.
%\end{myeg}

\bibliography{quiver}\bibliographystyle{alpha}

\end{document}